\documentclass[titlepage]{amsart}
\usepackage{amsaddr}

\usepackage{amssymb}
\usepackage{amscd}

\usepackage{amsmath}
\usepackage[all,cmtip,arrow]{xy}
\usepackage{amsthm}

\usepackage{amsfonts}
\usepackage{amscd}
\usepackage[OT2,T1]{fontenc}

\newtheorem{thm}{Theorem}[section]
\newtheorem{prop}[thm]{Proposition}
\newtheorem{lemma}[thm]{Lemma}

\newtheorem{lemma-def}[thm]{Lemma-Definition}
\newtheorem{theorem}{Theorem}[section]

\newtheorem{definition}[thm]{Definition}
\newtheorem{remark}[thm]{Remark}

 \font\Russ=wncyr10   scaled\magstep 1

\def\Sha{\hbox{\Russ\char88}}

\newcommand{\Zp}{{\mathbb{Z}_p}}
\newcommand{\Qp}{{\mathbb{Q}_p}}
\newcommand{\ZpG}{\mathbb{Z}_p[G]}

\newcommand{\CpG}{\mathbb{C}_p[G]}

\newcommand{\Cp}{{\mathbb{C}_p}}
\newcommand{\Ze}{{\mathbb{Z}}}
\newcommand{\Ce}{{\mathbb{C}}}

\newcommand{\Qu}{\mathbb{Q}}
\newcommand{\QG}{\Qu[G]}
\newcommand{\calP}{\mathcal{P}}

\newcommand{\QQ}{\mathbb Q}

\newcommand{\ZZ}{\mathbb Z}
\newcommand{\RR}{\mathbb R}
\newcommand{\CC}{\mathbb C}

\newcommand{\bz}{\mathbb Z}

\newcommand{\ZG}{{\bz [G]}}

\newcommand{\calI}{\mathcal I}

\newcommand{\calL}{\mathcal L}

\newcommand{\F}{\mathbb F}

\newcommand{\IpG}{{I}_p(G)}

\newcommand{\lra}{\longrightarrow}
\newcommand{\sseq}{\subseteq}
\newcommand{\Fp}{{\mathbb{F}_p}}

\usepackage[usenames]{color}

\DeclareMathOperator{\ind}{ind} 
\DeclareMathOperator{\res}{res} 
 
 \DeclareMathOperator{\im}{im}
 \DeclareMathOperator{\Gal}{Gal}

\DeclareMathOperator{\Hom}{Hom} \DeclareMathOperator{\End}{End}
\DeclareMathOperator{\Ext}{Ext}

\DeclareMathOperator{\Tr}{Tr}

\DeclareMathOperator{\Sel}{\rm Sel}
\DeclareMathOperator{\Map}{\rm Map}
\DeclareMathOperator{\ord}{ord}

\renewcommand{\det}{\text{det}}

\def\YEAR{\year}\newcount\VOL\VOL=\YEAR\advance\VOL by-1995
\def\firstpage{1}\def\lastpage{1000}
\def\received{}\def\revised{}
\def\communicated{}

\makeatletter
\def\magnification{\afterassignment\m@g\count@}
\def\m@g{\mag=\count@\hsize6.5truein\vsize8.9truein\dimen\footins8truein}
\makeatother

\font\eightrm=cmr8
\font\caps=cmcsc10
\font\Caps=cmcsc10 scaled \magstep1   

\makeatletter
\@specialpagefalse\headheight=8.5pt
\def\DocMath{}
\renewcommand{\@evenhead}{%
    \ifnum\thepage>\lastpage\rlap{\thepage}\hfill%
    \else\rlap{\thepage}\slshape\leftmark\hfill{\caps\SAuthor}\hfill\fi}%
\renewcommand{\@oddhead}{%
    \ifnum\thepage=\firstpage{\DocMath\hfill\llap{\thepage}}%
    \else{\slshape\rightmark}\hfill{\caps\STitle}\hfill\llap{\thepage}\fi}%
\makeatother

\def\TSkip{\bigskip}
\newbox\TheTitle{\obeylines\gdef\GetTitle #1
\ShortTitle  #2
\SubTitle    #3
\Author      #4
\ShortAuthor #5
\EndTitle
{\setbox\TheTitle=\vbox{\baselineskip=20pt\let\par=\cr\obeylines%
\halign{\centerline{\Caps##}\cr\noalign{\medskip}\cr#1\cr}}%
        \copy\TheTitle\TSkip\TSkip%
\def\next{#2}\ifx\next\empty\gdef\STitle{#1}\else\gdef\STitle{#2}\fi%
\def\next{#3}\ifx\next\empty%
    \else\setbox\TheTitle=\vbox{\baselineskip=20pt\let\par=\cr\obeylines%
    \halign{\centerline{\caps##} #3\cr}}\copy\TheTitle\TSkip\TSkip\fi%
\centerline{\caps #4}\TSkip\TSkip%
\def\next{#5}\ifx\next\empty\gdef\SAuthor{#4}\else\gdef\SAuthor{#5}\fi%
\ifx\received\empty\relax
    \else\centerline{\eightrm Received: \received}\fi%
\ifx\revised\empty\TSkip%
    \else\centerline{\eightrm Revised: \revised}\TSkip\fi%
\ifx\communicated\empty\relax
    \else\centerline{\eightrm Communicated by \communicated}\fi\TSkip\TSkip%
\catcode'015=5}}\def\Title{\obeylines\GetTitle}
\def\Abstract{\begingroup\narrower
    \parskip=\medskipamount\parindent=0pt{\caps Abstract. }}
\def\EndAbstract{\par\endgroup\TSkip}

\long\def\MSC#1\EndMSC{\def\arg{#1}\ifx\arg\empty\relax\else
     {\par\narrower\noindent%
     1991 Mathematics Subject Classification: #1\par}\fi}

\long\def\KEY#1\EndKEY{\def\arg{#1}\ifx\arg\empty\relax\else
        {\par\narrower\noindent Keywords and Phrases: #1\par}\fi\TSkip}

\newbox\TheAdd\def\Addresses{\vfill\copy\TheAdd\vfill
    \ifodd\number\lastpage\vfill\eject\phantom{.}\vfill\eject\fi}
{\obeylines\gdef\GetAddress #1
\Address #2
\Address #3
\Address #4
\EndAddress
{\def\xs{4.3truecm}\parindent=0pt
\setbox0=\vtop{{\obeylines\hsize=\xs#1\par}}\def\next{#2}
\ifx\next\empty 
     \setbox\TheAdd=\hbox to\hsize{\hfill\copy0\hfill}
\else\setbox1=\vtop{{\obeylines\hsize=\xs#2\par}}\def\next{#3}
\ifx\next\empty 
     \setbox\TheAdd=\hbox to\hsize{\hfill\copy0\hfill\copy1\hfill}
\else\setbox2=\vtop{{\obeylines\hsize=\xs#3\par}}\def\next{#4}
\ifx\next\empty\ 
     \setbox\TheAdd=\vtop{\hbox to\hsize{\hfill\copy0\hfill\copy1\hfill}
                \vskip20pt\hbox to\hsize{\hfill\copy2\hfill}}
\else\setbox3=\vtop{{\obeylines\hsize=\xs#4\par}}
     \setbox\TheAdd=\vtop{\hbox to\hsize{\hfill\copy0\hfill\copy1\hfill}
                \vskip20pt\hbox to\hsize{\hfill\copy2\hfill\copy3\hfill}}
\fi\fi\fi\catcode'015=5}}\gdef\Address{\obeylines\GetAddress}

\begin{document}
\Title

Congruences for critical values of higher derivatives
of twisted Hasse-Weil $L$-functions, III



\ShortTitle
Congruences for derivatives of Hasse-Weil $L$-functions, III
%


\SubTitle
\Author Werner Bley and Daniel Macias Castillo

\ShortAuthor Werner Bley and Daniel Macias Castillo
\EndTitle

\Abstract 
Let $A$ be an abelian variety defined over a number field $k$, let $p$ be an odd prime number and let $F/k$ be a cyclic extension of $p$-power degree. Under not-too-stringent hypotheses we give an interpretation of the $p$-component of the relevant case of the equivariant Tamagawa number conjecture in terms of integral congruence relations involving the evaluation on appropriate points of $A$ of the ${\rm Gal}(F/k)$-valued height pairing of Mazur and Tate. We then discuss the numerical computation of this pairing, and in particular obtain the first numerical verifications of this conjecture in situations in which the $p$-completion of the Mordell-Weil group of $A$ over $F$ is not a projective Galois module.
\EndAbstract

\Address

Werner Bley,
Ludwig-Maximilians-Universit\"at M\"unchen,
Theresienstr. 39,
D-80333 M\"unchen,
Germany,
bley@math.lmu.de

\Address
Daniel Macias Castillo,
Departamento de Matem\'aticas, 
Universidad Aut\'onoma de Madrid, 28049 Madrid (Spain);
and Instituto de Ciencias Matem\'aticas, 28049 Madrid (Spain).
daniel.macias@uam.es

\Address
\Address
\EndAddress
\section{Introduction}

Let $A$ be an abelian variety defined over a number field $k$. The Birch and Swinnerton-Dyer conjecture for $A$ over $k$ (as extended to this setting by Tate) predicts a remarkable equality between the leading term $L^*(A,1)$ at $z=1$ of the Hasse-Weil $L$-series $L(A,z)$ of $A$ over $k$  (assuming that this function has a suitable meromorphic continuation) and the key algebraic invariants of $A$ over $k$.

Nevertheless, there are various natural contexts in which it seems likely that this equality does not encompass the full extent of the interplay between the leading terms $L^*(A,\psi,1)$ at $z=1$ of the twisted Hasse-Weil $L$-series $L(A,\psi,z)$, associated to $A$ and to finite dimensional complex characters $\psi$ of the absolute Galois group of $k$, and the algebraic invariants of $A$. For instance, building on a conjecture due to Deligne and Gross concerning the order of vanishing at $z=1$ of such functions one may, for a fixed character $\psi$, predict that a suitable normalisation of $L^*(A,\psi,1)$ is algebraic and generates an explicit fractional ideal inside any large enough number field. See \cite[Prop. 7.3]{rbsd} for such explicit predictions.

However, even such conjectural formulas would not themselves account for any connections that might exist between the numbers $L^*(A,\psi,1)$ as $\psi$ ranges over characters that are not necessarily in the same Galois orbit. In this direction, Mazur and Tate \cite{mt} have, in certain concrete settings and by building on the theory of modular symbols, predicted explicit congruence relations between the {\it values} $L(A,\psi,1)$ at $z=1$ of the functions $L(A,\psi,z)$. In addition, Darmon \cite{darmon} has subsequently used the theory of Heegner points to formulate analogous predictions for the values $L'(A,\psi,1)$ at $z=1$ of the {\it first derivatives} $L'(A,\psi,z)$ of the functions $L(A,\psi,z)$.

Let $F$ be a finite Galois extension of $k$. Then, in all cases, the congruence relations discussed in the previous paragraph, as $\psi$ ranges over complex irreducible characters of ${\rm Gal}(F/k)$, involve the evaluation at suitable points of $A(k)$ of the canonical ${\rm Gal}(F/k)$-valued height pairings that had been previously constructed by Mazur and Tate \cite{mtbi} by using the geometrical theory of biextensions.

In recent work \cite{rbsd} of Burns and the second author, a completely general framework for the conjectural theory of integral congruence relations between the leading terms $L^*(A,\psi,1)$ has been developed, thereby extending and refining the aforementioned conjectures of Mazur and Tate and of Darmon. This framework relies on the formulation of a completely general `refined conjecture of Birch and Swinnerton-Dyer type' (or `refined BSD conjecture' in the sequel) for $A$ and $F/k$, which is then shown to encode general families of integral congruence relations involving the Mazur-Tate pairing.

Fix $F$ as above and set $G:={\rm Gal}(F/k)$. 
We let $A_F$ denote the base change of $A$ through $F/k$ and consider
$M_F := h^1(A_F)(1)$ as a motive over $k$ with a natural action of the semisimple $\Qu$-algebra $\QG$.

It is then also shown in \cite{rbsd} that the refined BSD conjecture is equivalent to the equivariant Tamagawa number conjecture (or `eTNC' in the sequel) for the pair $(M_F, \ZG)$, as formulated by Burns and Flach \cite{bufl99}.
In addition, both of these conjectures decompose naturally into `$p$-components', one for each rational prime number $p$, and each such component is itself of interest.

For example, if $A$ has good ordinary reduction at $p$, then the compatibility results proved by Burns and Venjakob \cite{BV2} show that this $p$-component is (under certain hypotheses) a consequence of the main conjecture of non-commutative Iwasawa theory for $A$, as formulated by Coates et al. \cite{cfksv}.

Assume now that $p$ is a fixed odd prime and $F/k$ a cyclic $p$-extension of degree $p^n$ for some natural number $n$.
In this note we will continue the study of the $p$-component of the eTNC that we begun in \cite{bleymc}. We recall that the main result of loc.cit. was the computation of an equivariant regulator under certain not-too-stringent conditions. Through this computation it was then possible to 
give a reformulation of this $p$-component which both was of theoretical interest and also made it
amenable to providing partial numerical evidence.

For simplicity of the exposition, let us in the rest of this introduction assume that $A$ is an elliptic curve. The computation of the equivariant regulator in \cite{bleymc} relied on the fact that, under suitable hypotheses, a representation-theoretic result due to Yakovlev \cite{yakovlev} implies that the $p$-completion $A(F)_p:=\ZZ_p\otimes_\ZZ A(F)$ of the Mordell-Weil group of $A$ over $F$ is a permutation $\ZZ_p[G]$-module. Explicitly speaking, this fact means that one may find points
$P_{(H,j)}\in A(F^H)$, where $H$ runs over all subgroups of $G$, with the property that $\ZZ_p[G/H]\cdot P_{(H,j)}$ is a free $\ZZ_p[G/H]$-module of rank one and also that
\[
  A(F)_p = \bigoplus_{H\leq G}\bigoplus_{j=1}^{m_H}\ZZ_p[G/H]\cdot P_{(H,j)}
\]
(for some set of non-negative integers $\{m_H:H\leq G\}$).

However, our formula for the equivariant regulator involved a choice of an integral matrix $\Phi$, with entries in $\ZZ_p[G]$, which depended upon a canonical extension class in the Yoneda 2-extension group ${\rm Ext}^2_{\ZZ_p[G]}(\Hom_{\ZZ_p}(A(F)_p,\ZZ_p),A(F)_p)$ and was therefore not computable in any examples unless this group vanished. Given the above direct sum decomposition of $A(F)_p$, the vanishing of this group holds if and only if $m_H$ is equal to zero for each subgroup $H\neq 1$,
or equivalently, if and only if $A(F)_p$ is a free $\ZZ_p[G]$-module of the form $\bigoplus_{j=1}^{m_1}\ZZ_p[G]\cdot P_{(1,j)}$.

This limitation of our previous methods is consistent with those occurring in all existing verifications of $p$-components of the eTNC
for any elliptic curves. Indeed, in the settings of the theoretical verifications obtained by the first author in
\cite{bleythree}, by Burns, Wuthrich and the second author in \cite{bmw},
or of the recent extensions of these results by Burns and the second author in \cite{rbsd},
as well as of the numerical verifications carried out in \cite{bmw,bleyone, bleytwo}
and in our previous article \cite{bleymc}, a full verification of this conjecture was only ever achieved in
situations which forced the $\ZZ_p[G]$-module $A(F)_p$ to be projective.

On the other hand, even for $n=1$ (meaning that the extension $F/k$ has degree $p$), the result \cite[Thm. 9.11]{rbsd} shows that the $p$-component of the eTNC (or refined BSD conjecture) encodes a family of congruence relations between the leading terms $L^*(A,\psi,1)$ and certain `Mazur-Tate regulators' (coming from the evaluation of Mazur-Tate height pairings) which are non-trivial unless $A(F)_p$ is projective.

This observation is consistent with our previously encountered difficulties and also justifies why, from the point of view of our approach, it is only interesting to consider components of the general conjectures at primes which divide the degree of the extension.



In this note we use a result of Burns and the second author \cite[Thm. 10.3]{rbsd} to obtain an alternative computation of the equivariant regulator which is much better suited for the purpose of verifying `non-projective' instances of the refined BSD conjecture.

To be a little more precise we note that, under our hypotheses, and for any subgroup $H$ of $G$, the Mazur-Tate pairing for $A$ considered over the sub-extension $F/F^H$ of $F/k$ gives a well defined pairing
\[
\langle\,,\rangle_{F/F^H}^{\rm MT}:A(F^H)_p\otimes_{\ZZ_p}A(F^H)_p \to H\cong
I_p(H)/I_p(H)^2 .
\]
Here $I_p(H)$ denotes the augmentation ideal in the group ring $\ZZ_p[H]$ and the isomorphism maps $g\in H$ to the class of $g-1$.

Fix any choice of points as above and any generator $\sigma$ of $G$.
For any double-indices $(H,j)$ and $(J,i)$ with $H,J\neq 1$ and $H\leq J$, we let $\Psi_{(H,j),(J,i)}$ be any element of $\ZZ_p[G]$ which, in the group
$$\ZZ_p[G/H]\otimes_{\ZZ_p}I_p(H)/I_p(H)^2,$$
satisfies the equality
\begin{equation*}
  \,\,\,\,\,\,\,\,\,\,\,\,\,\,\,\,\Psi_{(H,j),(J,i)}\otimes(\sigma^{|G/H|}-1)=\sum_{\gamma\in G/H}\bigl(\gamma\otimes
  \langle P_{(J,i)},\gamma P_{(H,j)}\rangle_{F/F^H}^{\rm MT}\bigr).
\end{equation*}
For double-indices with $J<H$ we will let $\Psi_{(H,j),(J,i)}$ be any element of $\ZZ_p[G]$ which satisfies a straightforward variant of this equality. See (\ref{definingpsilower}) and Remark \ref{bksremark} below for more details.

We then show that any matrix $\Psi$ obtained in this manner (by ordering all double-indices $(H,j),(J,i)$ lexicographically) is,
independently of all of the above choices, a suitable replacement for the essentially inexplicit matrix $\Phi$ that occurred in
the computation of the equivariant regulator in \cite{bleymc}.

In the main result of this note, Theorem \ref{main}, we thus give a reformulation for the $p$-component of the refined BSD
conjecture in terms of Mazur-Tate regulators obtained from considering natural components of the matrix $\Psi$.
This is in particular a suitable extension to general $n$ of the result \cite[Thm. 9.11]{rbsd} of Burns and the second author.


As an application we are
now able to obtain the first numerical verifications of the $p$-component of the refined BSD conjecture in situations in which the $p$-completed Mordell-Weil group $A(F)_p$ is not a projective $\ZZ_p[G]$-module. We emphasise again that there exists no other theoretical or numerical verification for this conjecture in such situations.

We shall give a detailed description, which we feel may be of some independent interest, of the methods that are appropriate to the
numerical computation of Mazur-Tate pairings. We comment upon the results of our computations in
  Section \ref{computational results}. Here
  we also give a list of  pairs $(A,F/k)$ for which we have numerically verified the
  conjecture, although this list is not exhaustive. See also the webpage of the first author for more details and the
  MAGMA implementation.

\subsection{General Notation}

For a finite abelian group $\Gamma$ we set
${\rm Tr}_\Gamma:=\sum_{\gamma \in \Gamma} \gamma\in\ZZ[\Gamma]$ and also
$\widehat{\Gamma}:=\Hom_\ZZ(\Gamma,\CC^\times)$. We write $\check\psi$ for the contragrediant character of each
$\psi \in \hat{\Gamma}$ and also write
\[
e_\psi = \frac{1}{|\Gamma|} \sum_{\gamma \in \Gamma} \psi(\gamma) \gamma^{-1}
\]
for the associated idempotent.


For any abelian group $M$ we let $M_{{\rm tor}}$ denote its torsion subgroup. 
We also set $M_p := \Zp \otimes_\Ze M$.
If $M$ is finitely generated, then for a field extension $E$ of $\Qu$ we shall sometimes abbreviate $E \otimes_\Ze M$ to $E \cdot M$.
Finally, for any integer $n$ we write $M[n]$ for the subgroup of $n$-torsion points of $M$. 

For any $\Zp[\Gamma]$-Module $M$ we write 
  $M^*$ for the linear dual $\Hom_\Zp(M, \Zp)$, endowed with the natural contragredient action of $\Gamma$.
Explicitly, for a homomorphism $f$ and elements $m \in M$ and $\gamma \in \Gamma$,
one has $(\gamma f)(m) = f(\gamma^{-1} m)$. If $\Delta$ is a subgroup of $\Gamma$, we write $M_\Delta$ for the module of $\Delta$-coinvariants of $M$.

For any Galois extension of fields we abbreviate $\Gal(L/K)$ to $G_{L/K}$. We fix an algebraic closure $K^c$ of $K$ and
abbreviate $G_{K^c/K}$ to $G_K$.


If $A$ is an abelian variety defined over a number field $k$ and $L/k$ a finite extension, we write $A(L)$
for the Mordell-Weil group, $\Sha(A_L)$ for the Tate-Shafarevich group of $A$ over $L$ and $\Sha_p(A_L)$ for its $p$-primary part.

\section{Statement of the main result}

In this section we state our standing hypotheses and, after defining all the relevant objects, state the main result of this article.

\subsection{The hypotheses}\label{hyp}

Let $A$ be an abelian variety of dimension $d$, defined over a number field $k$. Let $p$ be an odd prime number and let $F/k$ be a cyclic field extension of degree $p^n$, for some natural number $n$. We write $A^t$ for the dual abelian
variety.

As in \cite[Sec.~2]{bleymc}, we will assume the validity of the following list of hypotheses.

\begin{itemize}
\item [(a)] $p \nmid |A(k)_{\rm tor}| \cdot  |A^t(k)_{\rm tor}|$;
\item [(b)]  $p$ does not divide the Tamagawa number of $A$ at any place of $k$ at which it has bad reduction;
\item [(c)] $A$ has good reduction at all $p$-adic places of $k$;
\item [(d)] $p$ is unramified in $F/\QQ$;
\item [(e)] no place of bad reduction of $A$ is ramified in $F/k$;
\item [(f)] if a place $v$ of $k$ ramifies in $F/k$ then no point of order $p$ of the reduction of $A$ is defined over the residue field of $v$;
\item [(g)] $\Sha(A_{F})$ is finite;
\item [(h)] $\Sha_p(A_{F^H})$ vanishes for all non-trivial subgroups $H$ of $G$;
\item[(i)] The group $H^1(\Gal(k(A[p^n])/k),A[p^n])$ vanishes.
\end{itemize}

\begin{remark}{\em

The hypotheses (a)-(h) recover those in place throughout \cite{bleymc}. The full list of hypotheses also recovers those that are in place in \cite[Thm. 9.9, Thm.9.11]{rbsd}.
We refer the reader to \cite[Rem.~2.1]{bleymc} or \cite[Rem.~6.1]{rbsd} for a further discussion of these hypotheses.

The hypothesis (i) recovers Hypothesis 10.1 from \cite{rbsd} in our setting and will hence allow us to apply Theorem 10.3 of loc. cit.. We recall that it is widely satisfied. For instance, it holds whenever multiplication-by-`$-1$' belongs to the image of the canonical Galois representation $G_k\to{\rm Aut}_{\mathbb{F}_p}(A[p])$. In particular, if $A$ is an elliptic curve then this hypothesis excludes only finitely many primes, by a result of Serre. Moreover, if $A$ is an elliptic curve and $k$ does not contain any $p$-th roots of unity, Lawson and Wuthrich have recently shown that hypothesis (i) is valid in all but certain exceptional cases that, in particular, all have $p\leq 11$ (see \cite[Thm. 2, \S 6]{lw}).
}\end{remark}

\begin{remark}{\em It is possible, using computations in \cite{rbsd}, to obtain a generalisation of our main result under a significantly weaker version of hypotheses (d) (that still ensures the surjectivity of the appropriate norm maps on Mordell-Weil groups).}\end{remark}


Our main result will concern an `equivariant regulator' ${\rm Reg}_{A,F/k,j}$ in $\CC_p[G]^\times/\ZZ_p[G]^\times$
for each isomorphism $j:\CC\cong\CC_p$, which we now proceed to define. Throughout the construction of this element,
we will always use $j$ to implicitly identify $\widehat G$ with $\Hom_\ZZ(G,\CC_p^\times)$.

\subsection{N\'eron-Tate regulators}

For $0\leq r\leq n$ we denote by $J_r$ the subgroup of $G$ of order $p^{n-r}$ and set $F_r:=F^{J_r}$ and $\Gamma_r:=G/J_r$.
Clearly, $[F_r:k] = |\Gamma_r| = p^r$.

Our hypotheses  allow us to apply a result of Yakovlev \cite{yakovlev} in order to restrict the Galois structure 
of the Mordell-Weil groups $A(F)_p$ and $A^t(F)_p$ as follows.
For any natural number $m$ we write $[m]$ for the set $\{1,\ldots,m\}$.

By \cite[Prop.~2.2 and (2.1)]{bleymc} we may and will fix a set of non-negative integers $\{m_{r}:0\leq r\leq n\}$ and subsets 
\begin{eqnarray*}
\calP_{(r)} &=& \{P_{(r,j)}:j\in[m_r]\} \sseq A(F_r), \\
\calP^t_{(r)}&=&\{P^t_{(r,j)}:j \in[m_r]\}\sseq  A^t(F_r)
\end{eqnarray*}
such that the $\ZZ_p[\Gamma_r]$-modules generated by each of the points $P_{(r,j)}$ and $P^t_{(r,j)}$ are free of rank one 
and there are direct sum decompositions of $\ZZ_p[G]$-modules
\begin{equation}\label{global points}
 A(F)_p = \bigoplus_{r=0}^{n}\bigoplus_{j=1}^{m_r}\ZZ_p[\Gamma_r]\cdot P_{(r,j)} \,\,
\text{ and }\,\, A^t(F)_p = \bigoplus_{r=0}^{n}\bigoplus_{j=1}^{m_r}\ZZ_p[\Gamma_r]\cdot P^t_{(r,j)}.
\end{equation}
We set
\[
\calP := \bigcup_{r=0}^n \calP_{(r)}, \quad \calP^t := \bigcup_{r=0}^n \calP^t_{(r)}.
\]

%

By ordering our fixed choice of points $\calP$ and $\calP^t$ lexicographically, we obtain a `regulator matrix'
\[
R_{A,F/k}^{\rm NT}(\calP, \calP^t):=\left( \sum_{g \in G}\langle gP^t_{(r,j)},P_{(s,i)}\rangle_{A_F}\cdot g^{-1} \right)_{(r,j),(s,i)}
\in M_N(\RR[G])
\]
with $N := \sum_{r=0}^n m_r$, and where $\langle -,-\rangle_{A_F}$ denotes the N\'eron-Tate height pairing for $A$ over $F$.


For any matrix $X=(x_{(r,j), (s,i)})$ where the indices  $\{ (r,j) : 0\leq r \leq n, j\in[m_r] \}$
are ordered lexicographically and for any $0\leq t\leq n$, we set 
\[
X_t:=(x_{(r,j),(s,i)})_{r,s\geq t}
\]
and, given any matrix $Y$ with entries in $\CC[G]$, resp. $\CC_p[G]$, and any $\psi\in\widehat G$, we write $\psi(Y)$ for the matrix with entries in $\CC$, resp. $\CC_p$, obtained after extending $\psi$ to a function on $\CC[G]$, resp. $\CC_p[G]$, by linearity and then evaluating $\psi$ at each entry of $Y$. For any $\psi\in\widehat{G}$ we define an integer $t_\psi$ between $0$ and $n$ by the equality
  $\ker(\psi) = J_{t_\psi}$ and then set
  \begin{equation}\label{def psi minor}
    \varepsilon_\psi(Y) := \det\left( \psi(Y)_{t_\psi} \right).
  \end{equation}

We now put
\[
m_{\psi}:=\sum_{r=t_\psi}^{n}(n-r)m_r
\] 
and define the $\psi$-component of the equivariant N\'eron-Tate regulator by
\begin{equation}\label{def psi reg}
{\rm Reg}^{\rm NT}_\psi(\calP, \calP^t) :=p^{-2m_{\psi}}{\cdot\varepsilon_\psi\bigl(R_{A,F/k}^{\rm NT} (\calP, \calP^t)\bigr).} 
\end{equation}
Each regulator term ${\rm Reg}^{\rm NT}_\psi(\calP, \calP^t)$ coincides with the 
element $\lambda_\psi(\calP, \calP^t)$ defined in \cite[Def.~2.4]{bleymc}.

We finally fix a generator $\sigma$ of $G$  and then define a non-zero complex number 
\begin{equation}\label{def delta}
\delta_\psi:=\prod_{r=0}^{t_\psi-1}\left(\psi(\sigma)^{p^r}-1\right)^{m_r}.
\end{equation}

\subsection{Mazur-Tate regulators}\label{mtregs}

For any $0 \le r \le n$ we recall that, under our given hypotheses, \cite[Prop. 6.3(ii)]{rbsd} implies that
every element of $A^t(F_r)_p$ and $A(F_r)_p$ is `locally-normed'. Indeed, from this result one knows that for every finite prime $w$ of $F$, the $G$-modules $\ZZ_p\otimes_\ZZ A^t(F_w)$ and $\ZZ_p\otimes_\ZZ A(F_w)$ are cohomologically-trivial and from this it follows easily that the $p$-completions of the subgroups of locally-normed elements in $A^t(F_r)$ and in $A(F_r)$ are equal to the full $p$-completions $A^t(F_r)_p$ and $A(F_r)_p$, respectively.

In particular, the construction of
Mazur and Tate using the theory of biextensions gives well defined canonical height pairings
\begin{equation}\label{mtpairing}
\left\langle\,,\right\rangle_{F/F_r}^{\rm MT}:A^t(F_r)_p\otimes_{\ZZ_p}A(F_r)_p \to J_r\cong
I_p(J_r)/I_p(J_r)^2.
\end{equation}
Here $I_p(J_r)$ denotes the augmentation ideal in the group ring $\ZZ_p[J_r]$ and the isomorphism maps $g\in J_r$ to the class of $g-1$.

In the sequel we also write $\rho_r$ for the canonical projection $\ZZ_p[G]\to\ZZ_p[\Gamma_r]$.
For any indices $0\leq r,s\leq n-1$, any $j\in[m_r]$ and any $i\in[m_s]$, we set $\ell:={\rm max}(r,s)$ and then fix any elements $\Psi_{(r,j),(s,i)}$ of $\ZZ_p[G]$ which satisfy the equality
\begin{equation}\label{definingpsilower}
  \rho_r(\Psi_{(r,j),(s,i)})\otimes(\sigma^{p^\ell}-1)=
  \sum_{\gamma\in\Gamma_r}\bigl(\gamma\otimes\langle P^t_{(s,i)},\gamma P_{(r,j)} \rangle_{F/F_\ell}^{\rm MT}\bigr)
\end{equation} in
$$\ZZ_p[\Gamma_r]\otimes_{\ZZ_p}I_p(J_\ell)/I_p(J_\ell)^2.$$
It will be clear from Proposition \ref{bocksteincomputation} below that such elements always exist.

\begin{remark}\label{bksremark}{\em Let $\calI_p(J_r)$ denote the ideal of $\ZZ_p[G]$ generated by $I_p(J_r)$. Then the inclusion $I_p(J_\ell)\subset I_p(J_r)$ induces a canonical inclusion
\begin{equation}\label{bksisom} \Zp[\Gamma_r] \otimes_\Zp I_p(J_\ell)/I_p(J_\ell)^2 \hookrightarrow \Zp[\Gamma_r] \otimes_\Zp I_p(J_r)/I_p(J_r)^2\cong\calI_p(J_r)/\calI_p(J_r)^2.\end{equation}
Here the isomorphism is canonical (see \cite[Prop. 4.9]{gm}). It is then clear that the left-hand side of the equality (\ref{definingpsilower}) coincides with the class of $\Psi_{(r,j),(s,i)}\cdot(\sigma^{p^\ell}-1)$ in the quotient $\calI_p(J_r)/\calI_p(J_r)^2$. 
}\end{remark}




Using these choices we construct a matrix
\begin{equation}\label{Psi mat}
\Psi(\calP, \calP^t):=\left(
    \begin{array}{cc|c}
      \left(\Psi_{(r,j),(s,i)} \right)_{r,s < n} & &
      \begin{array}{c}
        0\\\vdots\\0
      \end{array}\\
      \hline 
      \begin{array}{ccc} 0 & \hdots & 0 \end{array} && I_{m_{n}}
\end{array}
\right)
\end{equation} with entries in $\ZZ_p[G]$. Here $I_{m_n}$ denotes the identity $m_n\times m_n$ matrix.

For any $\psi\in\widehat G$, 
Lemma \ref{still to do} below will show that the element
\begin{equation}\label{independentsum}\sum_{\psi\in\widehat G}\varepsilon_\psi(\Psi(\calP,\calP^t))\cdot e_\psi\end{equation}
of $\CC_p[G]$ belongs to $\CC_p[G]^\times$ (so in particular each term $\varepsilon_\psi(\Psi(\calP,\calP^t))$ is non-zero) and is independent, up to multiplication by an element of $\ZZ_p[G]^\times$, of the choices made in (\ref{definingpsilower}).

\subsection{The equivariant regulator} We may now define our equivariant regulator. 
We again set $N:=\sum_{r=0}^{n}m_r$.

\begin{definition}{\em
For a fixed isomorphism $j \colon \Ce \to \Cp$, the class of
\begin{equation}\label{equivariantreg}
{\rm Reg}_{A, F/k, j} :=(-1)^{N-m_n}\cdot \sum_{\psi \in \widehat G} 
\frac{j({\rm Reg}^{\rm NT}_\psi(\calP, \calP^t) \cdot \delta_\psi)} {\varepsilon_{\psi}(\Psi(\calP,\calP^t))}\cdot e_\psi 
\end{equation}
in $\CpG^\times / \ZpG^\times$ is the ($p$-primary) equivariant regulator associated to $A$ and $F/k$.
}
\end{definition}

The claims made in Section \ref{mtregs} imply that ${\rm Reg}_{A, F/k, j}$
is independent of the choices of $\calP, \calP^t, \sigma$ and the matrix $\Psi(\calP, \calP^t)$.
By abuse of terminology we will often refer to any choice of ${\rm Reg}_{A, F/k, j}$ in $\CpG^\times$ itself as the equivariant regulator.

\begin{remark}{\em
The results of Proposition \ref{bocksteincomputation} and Lemma \ref{still to do} {below} combine to show 
that   ${\rm Reg}_{A, F/k, j}$ is precisely the element 
\[
\sum_{\psi \in \hat G} \lambda_\psi(\calP, \calP^t) \cdot \varepsilon_\psi(\Phi) \cdot \delta_\psi\cdot e_\psi \quad (\text{ mod } \ZpG^\times)
\]
which occurs in \cite[Th.~2.9]{bleymc}. The main new insight is that the elements $\varepsilon_\psi(\Phi) \in \Ce_p^\times$ which
depended upon an essentially unknown matrix $\Phi \in M_N(\ZpG)$ can be explicitly determined by (\ref{definingpsilower}).
}\end{remark}

\subsection{Statement of the result}

The refined Birch and Swinnerton-Dyer conjecture is an equality between analytic and algebraic invariants associated with
$A/k$ and $F/k$. We now briefly describe the analytic part refering the reader to \cite[Sec.~2]{bleymc} or \cite{rbsd} for
further details.

For each $\psi \in \widehat{G}$ we set
\[
  \calL_\psi^* = \calL_{A, F/k, \psi}^* :=
  \frac{L_{S_{\rm r}}^*(A, \check{\psi}, 1) \cdot \tau^*(\Qu,  \ind_k^\Qu(\psi))^d }{\Omega_A^\psi\cdot w_\psi^d} \in \Ce^\times,
\]
where
\begin{itemize}
\item[-]  $L^*_{S_r}(A, \psi, 1)$ is the leading term in the Taylor expansion at
$z=1$ of the $\psi$-twisted Hasse-Weil $L$-function $L_{S_r}(A,\psi,z)$ of $A$, truncated by removing the Euler factors corresponding to the set $S_{{\rm r}}$ of primes of $k$ which ramify in $F/k$;
\item[-] $\tau^*(\Qu, \ind_k^\Qu(\psi))$ is a suitably modified global Galois-Gauss sum;
\item[-] $\Omega_A^\psi\cdot w_\psi^d$ is a (suitably normalised) product of periods.
\end{itemize}

We finally set
\[
\calL^* = \calL^*_{A, F/k} := \sum_{\psi \in \widehat{G}} \calL^*_{A, F/k, \psi} e_\psi \in \Ce[G]^\times
\]
and note that the element $\calL^*$ defined immediately above \cite[Th.~6.5]{rbsd} specialises precisely to our definition.

Without any further mention we will always assume that the functions $L_{S_r}(A,\psi,z)$ have analytic continuation to $z=1$, so that the above term is well-defined. We also recall that Deligne and Gross have then predicted that these functions should vanish at $z=1$ exactly to order equal to the multiplicity with which the character $\psi$ occurs in the representation $\CC\otimes_\ZZ A^t(F)$ of $G$.

We now formulate the main result of this manuscript. For any $j:\CC\cong\CC_p$ we write $j_*$ for the associated map $\CC[G]^\times\to \CC_p[G]^\times$.

\begin{theorem}\label{main} Assume that the hypotheses (a)-(i) are valid.
Let $\calP$ and $\calP^t$ be any choice of points such that  (\ref{global points}) holds.  Assume also that $\Sha_p(A_F) = 0$.

Then the $p$-component of the refined Birch and Swinnerton-Dyer conjecture is valid if and only if, 
for any $j:\CC\cong\CC_p$, the element
\begin{equation}\label{mainclaim}
\frac{j_*\left( \calL^*_{A, F/k} \right)}{ {\rm Reg}_{A, F/k, j} }
\end{equation}
belongs to $\ZpG^\times$.
\end{theorem}

\begin{remark}\label{acta}{\em One may rephrase the condition that the element (\ref{mainclaim}) belongs to $\ZZ_p[G]^\times$ in terms of explicit congruence relations in the augmentation filtration, as occurring in \cite[Conj. 3.11]{dmc}.
}\end{remark}

The proof of Theorem \ref{main} will occupy the next section.

\section{The proof of Theorem \ref{main}}

{Under our listed hypotheses, Theorem 6.5 and Remark 6.6 in \cite{rbsd} reformulate the $p$-component of the refined BSD conjecture (denoted by ${\rm BSD}_p(A_{F/k})$ throughout loc. cit.) as an equality of the form
\begin{equation*}\label{rbsdp}\delta_{G,p}\left(j_*\left(\calL^*_{A,F/k}\right)\right)=\chi_{G,p}\left({\rm SC}_p(A_{F/k}),h_{A,F}^j\right)\end{equation*}
in the relative algebraic $K$-group $K_0(\ZZ_p[G],\CC_p[G])$. Here $$\delta_{G,p}:\CC_p[G]^\times\to K_0(\ZZ_p[G],\CC_p[G])$$ is the canonical `extended boundary homomorphism' considered by Burns and Flach in {\cite[Sec.~4.2]{bufl99}} while 
the right-hand side 
is the refined Euler characteristic of the pair $({\rm SC}_p(A_{F/k}),h_{A,F}^j)$.
We thus also recall that ${\rm SC}_p(A_{F/k})$ is the `classical $p$-adic Selmer complex' for $A$ and $F/k$, as defined in Definition 2.3 of \cite{rbsd}, while $h_{A,F}^j$ is the canonical trivialisation of this complex that is induced by the N\'eron-Tate height (see below).


At the outset we set $C:={\rm SC}_p(A_{F/k})$ and note that, under the hypotheses of Theorem \ref{main}, it is proved in \cite[Prop. 6.3]{rbsd} that $C$ is a perfect complex of $\ZZ_p[G]$-modules. In addition, there are canonical identifications
\begin{equation}\label{bkcohomology}
  r_i:H^i(C)\cong\begin{cases}A^t(F)_p, & i=1;\\ 
                              A(F)_p^*, & i=2;\\
                              0, & i \ne 1,2.
                            \end{cases}
\end{equation}
The trivialisation $$h_{A,F}^j:\CC_p\otimes_{\ZZ_p}H^1(C)\cong\CC_p\otimes_{\ZZ_p}H^2(C)$$ is then the canonical isomorphism induced
by the N\'eron-Tate height pairing via (\ref{bkcohomology}) and $j$.

We now rephrase the validity of the $p$-component of the refined BSD conjecture as the vanishing of the element
$$\xi:=\delta_{G,p}\left(j_*\left(\calL^*_{A,F/k}\right)\right)-\chi_{G,p}(C,h_{A,F}^j)$$
of $K_0(\ZZ_p[G],\CC_p[G])$.
}
  \begin{remark}\label{K0 remark}{\em
    For an abelian group $G$ and a field $E$ with $\Qp \sseq E \sseq \Cp$ the relative algebraic
    $K$-group $K_0(\ZZ_p[G],E[G])$ is naturally
  isomorphic to $E[G]^\times / \ZpG^\times$. Moreover, these isomorphisms are induced by $\delta_{G,p}$.
 } \end{remark}

The first key step is to explicitly compute the second term that occurs in the definition of $\xi$ and to do this we  use the canonical identifications (\ref{bkcohomology}) to identify the complex $C$ with a unique element $\delta_{A,F,p} = \delta_{C,r_1,r_2}$ of the Yoneda Ext-group ${\rm Ext}^2_{\ZZ_p[G]}(A(F)_p^*,A^t(F)_p)$.

For each pair $(r,j)$ we define a dual element $P_{(r,j)}^*$ of $A(F)_p^*$ by setting,
for any pair $(s,i)$ and $\tau\in G$, 
\begin{equation}\label{def P star}
P_{(r,j)}^*(\tau P_{(s,i)}):=\begin{cases} 1, & \text{ if } r=s, j=i \text{ and } \tau\in J_r;\\
                                        0, & \text{ otherwise }.
                          \end{cases}
\end{equation}
By \cite[Lem. 4.1]{bleymc} one then has 
\[
A(F)_p^*=\bigoplus_{(r,j)}\ZZ_p[\Gamma_r] P_{(r,j)}^*
\]
with each summand isomorphic to $\ZZ_p[\Gamma_r]$.
We fix a free $\ZZ_p[G]$-module 
\[
X:=\bigoplus_{(r,j)}\ZZ_p[G] b_{(r,j)}
\] 
of rank $N=\sum_r m_r$ and consider the exact sequence
\begin{equation}\label{syzygy}
0\to A^t(F)_p\stackrel{\iota}{\to}X\stackrel{\Theta}{\to}X\stackrel{\pi}{\to}A(F)_p^*\to 0,
\end{equation}
where we set
\[
\pi(b_{(r,j)}):=P_{(r,j)}^*, \quad \Theta(b_{(r,j)}):=(\sigma^{p^r}-1)b_{(r,j)} \text{ and } \iota(P^t_{(r,j)}):=\Tr_{J_r}b_{(r,j)}.
\]

This sequence defines a canonical isomorphism
\begin{equation}\label{computeext}
\Ext^2_{\ZZ_p[G]}(A(F)_p^*,A^t(F)_p)\cong\End_{\ZZ_p[G]}(A^t(F)_p)/\iota_*(\Hom_{\ZZ_p[G]}(X,A^t(F)_p))
\end{equation}
where $\iota_*$ denotes composition with $\iota$.
Before proceeding to describe the map (\ref{computeext}) let us note that it is bijective by the general result \cite[Thm. IV.9.1]{HS}.

For a given $\phi\in\End_{\ZZ_p[G]}(A^t(F)_p)$ we consider the 
push-out commutative diagram with exact rows
\begin{equation}\label{pushout}
\begin{CD} 0 @>  >> A^t(F)_p @> \iota >> X @> \Theta >> X @> \pi >> A(F)^*_p @>  >> 0\\
@. @V \phi VV  @V VV  @\vert @\vert @. \\
0 @>  >> A^t(F)_p @>  >> X(\phi) @>  >> X @> \pi >> A(F)^*_p @>  >> 0.
\end{CD}
\end{equation}
In this diagram $X(\phi)$ is defined as the push-out of $\iota$ and $\phi$ and all the unlabeled arrows are the canonical maps 
induced by the push-out construction. Then the pre-image of $\phi$ under (\ref{computeext}) is represented 
by the bottom row of this diagram.

Then, since $C$ belongs to $D^{\rm perf}(\ZZ_p[G])$, the results of \cite[Lem. 4.2 and Lem. 4.3]{bleymc}
imply that there exists an automorphism $\phi$ of the $\ZZ_p[G]$-module $A^t(F)_p$ that represents the
image of $\delta_{A,F,p}$ under (\ref{computeext}) and also fixes the element $P^t_{(n,j)}$ for every $j$ in $[m_n]$.

It follows that the exact sequence
\begin{equation}\label{yonedarep}0\to A^t(F)_p\stackrel{\iota\circ\phi^{-1}}{\to}X\stackrel{\Theta}{\to}X\stackrel{\pi}{\to}A(F)_p^*\to 0\end{equation}
is a representative of the extension class $\delta_{A,F,p}$.

In particular, for any choice of $\CC_p[G]$-equivariant splittings
\begin{equation}\label{splitting1}
  s_1:\CC_p\cdot X\to \CC_p\cdot A^t(F)_p\oplus\CC_p\cdot\im(\Theta)
\end{equation}
and
\begin{equation}\label{splitting2}
  s_2:\CC_p\cdot X\to \CC_p\cdot A(F)^*_p\oplus\CC_p\cdot\im(\Theta)
\end{equation}
of the scalar extensions of the canonical exact sequences
$$0\to A^t(F)_p\stackrel{\iota}{\to}X\stackrel{\Theta}{\to}\im(\Theta)\to 0$$
and
$$0\to\im(\Theta)\to X\stackrel{\pi}{\to}A(F)_p^* \to 0,$$
an explicit computation of the refined Euler characteristic occurring in the definition of $\xi$ implies that
\begin{align}\label{separatingstuff}&-\chi_{G,p}(C,h_{A,F}^j)\\
=&-\delta_{G,p}({\det}_{\CC_p[G]}(s_2^{-1}\circ(h_{A,F}^j\oplus{\rm id}_{\CC_p\cdot\im(\Theta)})\circ((\CC_p\cdot\phi)\oplus{\rm id}_{\CC_p\cdot\im(\Theta)})\circ s_1))\notag\\
=&-\delta_{G,p}({\det}_{\CC_p[G]}(s_2^{-1}\circ(h_{A,F}^j\oplus{\rm id}_{\CC_p\cdot\im(\Theta)})\circ s_1\circ s^{-1}_1\circ((\CC_p\cdot\phi)\oplus{\rm id}_{\CC_p\cdot\im(\Theta)})\circ s_1))\notag\\
=&\,\delta_{G,p}({\det}_{\CC_p[G]}(s_1^{-1}\circ((h^{j}_{A,F})^{-1}\oplus{\rm id}_{\CC_p\cdot\im(\Theta)})\circ s_2))\notag\\
&\hskip 1.5truein +\delta_{G,p}({\det}_{\CC_p[G]}(s_1^{-1}\circ((\CC_p\cdot\phi^{-1})\oplus{\rm id}_{\CC_p\cdot\im(\Theta)})\circ s_1))\notag\\
=&\,\delta_{G,p}({\det}_{\CC_p[G]}(s_1^{-1}\circ((h^{j}_{A,F})^{-1}\oplus{\rm id}_{\CC_p\cdot\im(\Theta)})\circ s_2))\notag\\
&\hskip 1.5truein +\delta_{G,p}({\det}_{\QQ_p[G]}((\QQ_p\cdot\phi^{-1})\oplus{\rm id}_{\QQ_p\cdot\im(\Theta)}))\notag.\end{align}

In addition, specialising the explicit computation of \cite[Prop. 4.4]{bleymc} to the case $\Phi={\rm id}_{A^t(F)_p}$ shows that
\begin{equation}\label{theregulatorterm}
  {\det}_{\CC_p[G]}(s_1^{-1}\circ( (h^{j}_{A,F})^{-1} \oplus{\rm id}_{\CC_p\cdot\im(\Theta)})\circ s_2)=
  \left(\sum_{\psi\in\widehat{G}}j({\rm Reg}_{\psi}^{\rm NT}(\calP,\calP^t)\cdot\delta_\psi)\cdot e_\psi\right)^{-1}.
\end{equation}

To compute the second term that occurs in the final equality of (\ref{separatingstuff}) we may and will fix
any elements $\Lambda_{(r,j), (s,i)}$ of $\ZZ_p[G]$ with the property that
\begin{equation}
\label{definingphiprime}
\phi^{-1}(P^t_{(s,i)}) = \sum_{(r,j)} \Lambda_{(r,j), (s,i)} P^t_{(r,j)}.
\end{equation}
We thus obtain an invertible matrix 
\[
\Lambda = \Lambda(\calP, \calP^t):=\left( \Lambda_{(r,j), (s,i)} \right)_{(r,j), (s,i)}
\] 
with entries $\Lambda_{(r,j), (s,i)}$ in $\ZZ_p[G]$ uniquely determined modulo the kernel of the canonical projection $\rho_r:\ZpG \to \Zp[\Gamma_r]$,  which also has the form (\ref{Psi mat}).

Recall the definition of $\varepsilon_\psi(\Lambda)$ in (\ref{def psi minor}).
The chosen properties of the representative $\phi$ of $\delta_{A,F,p}$ fixed above then imply that
\begin{equation}\label{leftovertermMT}{\det}_{\QQ_p[G]}((\QQ_p\cdot\phi^{-1})\oplus{\rm id}_{\QQ_p\cdot\im(\Theta)})=
\sum_{\psi\in\widehat{G}}\varepsilon_\psi(\Lambda)\cdot e_\psi.\end{equation}

The equalities (\ref{separatingstuff}), (\ref{theregulatorterm}) and (\ref{leftovertermMT}) now combine with the definition of $\xi$ to imply that
 $\xi=\delta_{G,p}(\mathcal{L})$ with
 \[
   \mathcal{L} :=j_*(\calL^*_{A,F/k})\cdot\left( \sum_{\psi\in\widehat{G}}
     j({\rm Reg}_{\psi}^{\rm NT}(\calP,\calP^t)\cdot\delta_\psi) \cdot e_\psi\right)^{-1}\cdot
   \left( \sum_{\psi\in\widehat{G}}\varepsilon_\psi(\Lambda)\cdot e_\psi\right).
 \]

 By Remark \ref{K0 remark} it thus follows that $\xi$ vanishes if and only if $\calL$ belongs to
 $\ZZ_p[G]^\times$ and hence the proof of Theorem \ref{main}
 is completed by the proposition below.

\begin{prop}\label{bocksteincomputation}
  For any $0\leq r,s\leq n-1$ we set $\ell:={\rm max}(r,s)$. Then for any $j\in[m_r]$, any $i\in[m_s]$ and any elements
  $\Lambda_{(r,j),(s,i)}$ of $\ZZ_p[G]$ satisfying (\ref{definingphiprime}) one has
\begin{equation*}\rho_r(\Lambda_{(r,j),(s,i)})\otimes(\sigma^{p^\ell}-1)=-\sum_{\gamma\in\Gamma_r}\bigl(\gamma\otimes\langle P^t_{(s,i)},\gamma(P_{(r,j)})\rangle_{F/F_\ell}^{\rm MT}\bigr)\end{equation*} in
$\ZZ_p[\Gamma_r]\otimes_{\ZZ_p}I_p(J_\ell)/I_p(J_\ell)^2$.
\end{prop}

\begin{proof}
We refer the reader to \cite[App.~B.2]{rbsd} for the construction of algebraic height pairings via Bockstein homomorphisms.
It is proved in \cite[Thm. 8.3]{rbsd} that 
\[
\langle\,,\rangle^{\rm MT}_{F/F_\ell} \colon A^t(F_\ell)_p\otimes_{\Zp} A(F_\ell)_p\to I_p(J_\ell)/I_p(J_\ell)^2 
\] 
coincides with the inverse of the pairing induced by 
\begin{align*}
\beta_{{A,F/F_\ell,p}} \colon A^t(F_\ell)_p\to I_p(J_\ell) / I_p(J_\ell)^2\otimes_{\Zp}(A(F)_p^*)_{J_\ell} \to  
I_p(J_\ell) / I_p(J_\ell)^2 \otimes_{\Zp}A(F_\ell)_p^*,
\end{align*} 
where the first arrow is the Bockstein homomorphism associated to the complex of $\ZZ_p[J_\ell]$-modules ${\rm SC}_p(A_{F/F_\ell})$ together with the canonical identifications (\ref{bkcohomology}), and the second arrow is induced by restriction to $A(F_\ell)_p$.

Now, it is immediately clear from their definitions in \cite[Def. 2.3]{rbsd} that the complexes ${\rm SC}_p(A_{F/F_\ell})$ and ${\rm SC}_p(A_{F/k})$ are canonically isomorphic in $D(\ZZ_p[J_\ell])$ and furthermore that this isomorphism is compatible with the identifications (\ref{bkcohomology}). The complex ${\rm SC}_p(A_{F/F_\ell})$ may therefore be represented by the (restriction of scalars of) the exact sequence (\ref{yonedarep}).

Now $\beta_{{A,F/F_\ell,p}}$ may be computed, through the representative (\ref{yonedarep}),
as the connecting homomorphism which arises when applying the 
snake lemma to the following commutative diagram (in which both rows and the third column are
exact and the first column is a complex)
\[
\begin{CD}
@. @. @. A^t(F_\ell)_p\\ @. @. @. @VV (\iota\circ\phi^{-1})^{J_\ell} V  \\ 0 @>
>> I_p(J_\ell)\otimes_{\ZZ_p[J_\ell]}X @> \subseteq >> X @> \Tr_{J_\ell}  >>
 X^{J_\ell}  @> >> 0\\
@. @VV {\rm id}\otimes_{\ZZ_p[J_\ell]}\Theta V @VV \Theta V @VV\Theta^{J_\ell} V\\
0 @>
>> I_p(J_\ell)\otimes_{\ZZ_p[J_\ell]}X @> \subseteq >> X @> \Tr_{J_\ell} >>
 X^{J_\ell}  @> >> 0\\
@. @VV ({\rm id}\otimes_{\ZZ_p[J_\ell]}\pi)_{J_\ell}  V \\ @.
I_p(J_\ell) / I_p(J_\ell)^2\otimes_{\Zp} (A(F)_p^*)_{J_\ell}.
\end{CD}
\]


From (\ref{definingphiprime}) and the definition of $\iota$ we immediately derive 
$$   \iota(\phi^{-1}(P^t_{(s,i)})) 
  =\iota\left(\sum_{(u,h)} \Lambda_{(u,h),(s,i)} P^t_{(u,h)}\right) 
  =\sum_{(u,h)} \Lambda_{(u,h),(s,i)} \Tr_{J_u}(b_{(u,h)}).
$$
By its definition in (\ref{def P star}) one has $P^*_{(u,h)}(\gamma P_{(r,j)})=0$ whenever the pair $(u,h)$ is different from $(r,j)$.
We thus find for any $\gamma\in\Gamma_r$ that
\begin{align*}
  -\langle P^t_{(s,i)},\gamma P_{(r,j)}\rangle_{F/F_\ell}^{\rm MT}=&\left(({\rm id}\otimes_{\ZZ_p[J_\ell]}\pi)_{J_\ell}\left(\Theta\left(\Tr_{J_r/J_\ell}(\Lambda_{(r,j),(s,i)} b_{(r,j)})\right)\right)\right)(\gamma P_{(r,j)})\\
=&\left(({\rm id}\otimes_{\ZZ_p[J_\ell]}\pi)_{J_\ell}\left((\sigma^{p^r}-1)\Tr_{J_r/J_\ell}(\Lambda_{(r,j),(s,i)} b_{(r,j)})\right)\right)(\gamma P_{(r,j)})\\
=&\left(({\rm id}\otimes_{\ZZ_p[J_\ell]}\pi)_{J_\ell}\left((\sigma^{p^\ell}-1)(\Lambda_{(r,j),(s,i)} b_{(r,j)})\right)\right)(\gamma P_{(r,j)})\\
  =&\left(\left((\sigma^{p^\ell}-1)+I_p(J_\ell)^2\right)\otimes
     (\Lambda_{(r,j),(s,i)}P^*_{(r,j)})\right)(\gamma P_{(r,j)})\\
=&(\sigma^{p^\ell}-1) ((\Lambda_{(r,j),(s,i)}P^*_{(r,j)})(\gamma P_{(r,j)}))+I_p(J_\ell)^2.
\end{align*}
Here the first equality uses the fact that
$$
\Tr_{J_r}(\Lambda_{(r,j),(s,i)} b_{(r,j)})=
\Tr_{J_\ell}(\Tr_{J_r/J_\ell}(\Lambda_{(r,j),(s,i)} b_{(r,j)})).
$$

Now, if we write
$\rho_r(\Lambda_{(r,j),(s,i)}) = \sum_{\gamma\in \Gamma_r} a_\gamma\gamma$ in $\ZZ_p[\Gamma_r]$, then
the definition (\ref{def P star}) of $P^*$ implies
\[
(\Lambda_{(r,j),(s,i)}P^*_{(r,j)})(\gamma P_{(r,j)})=a_\gamma.
\]

It therefore follows that the right-hand side of the claimed equality is equal to
$$\sum_{\gamma\in\Gamma_r}(\gamma\otimes(a_\gamma(\sigma^{p^\ell}-1)))=\sum_{\gamma\in\Gamma_r}((a_\gamma\gamma)\otimes(\sigma^{p^\ell}-1))=\rho_r(\Lambda_{(r,j),(s,i)})\otimes(\sigma^{p^\ell}-1),$$
as required.


\end{proof}

For the proof that all of our constructions are independent of any choices we made we will
  need the following general result which is straightforward to prove.

\begin{lemma}\label{ideal}
  For any indices $0\leq r\leq\ell\leq n$ and any elements $\lambda$ and $\lambda'$ of $\ZZ_p[G]$, one has that
  \[
    \rho_r(\lambda)\otimes(\sigma^{p^\ell}-1) = \rho_r(\lambda')\otimes (\sigma^{p^\ell}-1)
  \]
in $\ZZ_p[\Gamma_r]\otimes_{\ZZ_p}I_p(J_\ell)/I_p(J_\ell)^2$
if and only if $p^{\ell-r}(\lambda- \lambda')$ belongs to the ideal $$(\sigma^{p^r}-1)\cdot\ZZ_p[G]+\Tr_{J_r}\cdot\ZZ_p[G]$$ of $\ZZ_p[G]$.
\end{lemma}

We finally justify that all assertions in the statement of our main result were indeed well-defined.

\begin{lemma}\label{still to do}
  For any elements $\Psi_{(r,j),(s,i)}$ of $\ZZ_p[G]$ satisfying the equalities (\ref{definingpsilower}), the sum (\ref{independentsum}) belongs to $\CC_p[G]^\times$ and is independent,
  up to multiplication by an element of $\ZZ_p[G]^\times$, of the choices made.
\end{lemma}

\begin{proof}

By Proposition \ref{bocksteincomputation}, any collection of elements $-\Lambda_{(r,j),(s,i)}$ (for $r,s<n$) satisfying (\ref{definingphiprime}) also constitutes an appropriate choice satisfying the equalities (\ref{definingpsilower}). In addition, the sum
$$\sum_{\psi\in\widehat G}\varepsilon_\psi(-\Lambda)\cdot e_\psi$$
clearly belongs to $\CC_p[G]^\times$. (In fact, by the argument of \cite[Lem. 4.8]{bleymc}, it also belongs to $\mathcal{M}^\times$, where $\mathcal{M}$ denotes the (unique) maximal $\Zp$-order in $\QQ_p[G]$.)

It is therefore enough to set
$\Psi':=-\Lambda$, fix any collection of elements $\Psi_{(r,j),(s,i)}$ satisfying the equalities (\ref{definingpsilower}), and prove that the sum 
$$\sum_{\psi\in\widehat G}\frac{\varepsilon_\psi(\Psi)}{\varepsilon_\psi(\Psi')}\cdot e_\psi$$ belongs to $\ZZ_p[G]^\times$.


The main step involved in proving this assertion is given by the following intermediate result.

\begin{lemma}\label{factorsthrough} The endomorphism $\psi$ of $A^t(F)_p$ which maps a point $P^t_{(s,i)}$ to the sum
$$\sum_{(r,j)}(\Psi_{(r,j),(s,i)}-\Psi'_{(r,j),(s,i)})\cdot P^t_{(r,j)}$$ factors through the map $\iota:A^t(F)_p\to X$ occurring in the exact sequence (\ref{syzygy}).

In addition, if we define an endomorphism $\gamma$ of $A^t(F)_p$ by setting
$$\gamma(P^t_{(s,i)}):=\sum_{(r,j)}\Psi_{(r,j), (s,i)} P^t_{(r,j)},$$
then $\gamma$ is bijective.
\end{lemma}

\begin{proof}
  By Lemma \ref{ideal}, for each pair $(r,j),(s,i)$ with $r,s<n$, there exist elements $\lambda_{(r,j),(s,i)}$
  and $\mu_{(r,j),(s,i)}$ in $\ZZ_p[G]$ with the property that
\begin{equation}\label{difference}
    \Psi_{(r,j),(s,i)}-\Psi'_{(r,j),(s,i)}=\begin{cases}(\sigma^{p^r}-1)\cdot\lambda_{(r,j),(s,i)}+\Tr_{J_r}\cdot\mu_{(r,j),(s,i)},\,\,\,\,\,\,\,\,\,\,\,\,\,\,\,\,\,\,\,\,\,\,\,\text{ if }r\geq s;\\
\bigl((\sigma^{p^r}-1)\cdot\lambda_{(r,j),(s,i)}+\Tr_{J_r}\cdot\mu_{(r,j),(s,i)}\bigr)p^{r-s},\,\,\,\,\,\,\text{ if }r< s.
\end{cases}
\end{equation}
For $r=n$ or $s=n$ we may simply take
$\lambda_{(r,j),(s,i)}$ and $\mu_{(r,j),(s,i)}$ to be equal to 0 and still have these equalities.

Now since $\sigma^{p^r}$ acts trivially on $P^t_{(r,j)}$ one thus has
\begin{align}\label{headache}\psi(P^t_{(s,i)})=&\sum_{r\geq s}\Tr_{J_r}\cdot\mu_{(r,j),(s,i)}\cdot P^t_{(r,j)}+\sum_{r<s}p^{r-s}\Tr_{J_r}\cdot\mu_{(r,j),(s,i)}\cdot P^t_{(r,j)}\notag\\
=&\sum_{r\geq s}\Tr_{J_r}\cdot\mu_{(r,j),(s,i)}\cdot P^t_{(r,j)}+\sum_{r<s}\Tr_{J_s}\cdot\mu_{(r,j),(s,i)}\cdot P^t_{(r,j)}.\end{align}
In addition, $J_s$ acts trivially both on $P^t_{(s,i)}$ and on the second summand of the above expression and therefore it must also act trivially on each term $\Tr_{J_r}\cdot\mu_{(r,j),(s,i)}\cdot P^t_{(r,j)}$ with $r\geq s$.
Now, since $\ZZ_p[G]\cdot P^t_{(r,j)}=\ZZ_p[G/J_r]\cdot P^t_{(r,j)}$ is a free $\ZZ_p[G/J_r]$-module of rank one, this condition necessarily implies that $\Tr_{J_s/J_r}$ divides $\mu_{(r,j),(s,i)}$. We may therefore, for every indices with $r\geq s$, write $\mu_{(r,j),(s,i)}=\Tr_{J_s/J_r}\cdot\tilde{\mu}_{(r,j),(s,i)}$ for some element $\tilde{\mu}_{(r,j),(s,i)}$ of $\ZZ_p[G]$.

We now define a homomorphism $\alpha:X\to A^t(F)_p$ by setting
$$\alpha(b_{(s,i)}):=\sum_{r\geq s}\tilde{\mu}_{(r,j),(s,i)}\cdot P^t_{(r,j)}+\sum_{r<s}\mu_{(r,j),(s,i)}\cdot P^t_{(r,j)}$$
and claim that $\psi=\alpha\circ\iota$, as required to prove the first claim of the lemma.

Indeed, one has
\begin{align*}(\alpha\circ\iota)(P_{(s,i)})=&\alpha(\Tr_{J_s}\cdot P_{(s,i)})=\Tr_{J_s}\cdot\alpha(P_{(s,i)})\\
=&\sum_{r\geq s}\Tr_{J_s}\cdot\tilde{\mu}_{(r,j),(s,i)}\cdot P^t_{(r,j)}+\sum_{r<s}\Tr_{J_s}\cdot\mu_{(r,j),(s,i)}\cdot P^t_{(r,j)}\\
=&\sum_{r\geq s}\Tr_{J_r}\cdot\mu_{(r,j),(s,i)}\cdot P^t_{(r,j)}+\sum_{r<s}\Tr_{J_s}\cdot\mu_{(r,j),(s,i)}\cdot P^t_{(r,j)},\end{align*}
which is equal to $\psi(P_{(s,i)})$ by (\ref{headache}).

To prove the second claim we write $\epsilon:\ZZ_p[G]\to\ZZ_p$ for the canonical augmentation map and define a canonical ring homomorphism $\epsilon_{\F_p}$ as the composition
$$\ZZ_p[G]\stackrel{\epsilon}{\to}\ZZ_p\to\F_p.$$

Now, the equalities (\ref{difference}) imply that
$$\epsilon(\Psi'_{(r,j),(s,i)})=\begin{cases}\epsilon(\Psi_{(r,j),(s,i)})-p^{n-r}\epsilon(\mu_{(r,j),(s,i)}),\,\,\,\,\,\,\text{ if }r\geq s;\\
\epsilon(\Psi_{(r,j),(s,i)})-p^{n-s}\epsilon(\mu_{(r,j),(s,i)}),\,\,\,\,\,\,\text{ if }r< s,\end{cases}$$
and hence also that if $r,s<n$ then $\epsilon_{\F_p}(\Psi'_{(r,j),(s,i)})=\epsilon_{\F_p}(\Psi_{(r,j),(s,i)})$ (this equality is trivial for $r=n$ or $s=n$).

The ring $\ZZ_p[G]$ is local with maximal ideal equal to $\ker(\epsilon_{\F_p})$. In particular an element $x$ of $\ZZ_p[G]$ is a unit if and only if $\epsilon_{\F_p}(x)\neq 0$.

One then knows that $\epsilon_{\F_p}(\det(\Psi))=\det(\epsilon_{\F_p}(\Psi))=\det(\epsilon_{\F_p}(\Psi'))=\epsilon_{\F_p}(\det(\Psi'))\neq 0$. 
It thus follows that $\det(\Psi)\in\ZZ_p[G]^\times$ and therefore that $\gamma$ is bijective, as required.
\end{proof}

Now, if we denote by $[f]$ the pre-image in ${\rm Ext}^2_{\ZZ_p[G]}(A(F)_p^*,A^t(F)_p)$ under the isomorphism (\ref{computeext}) of the class of an endomorphism $f$ of $A^t(F)_p$,
then Lemma \ref{factorsthrough} implies that
$[\gamma]-[-\phi^{-1}]=[\psi]=0$ and hence also that
$[\gamma]=[-\phi^{-1}]$.

In addition, since both $\gamma$ and $-\phi^{-1}$ are bijective, this class in ${\rm Ext}^2_{\ZZ_p[G]}(A(F)_p^*,A^t(F)_p)$ may be represented by both of the exact sequences obtained by replacing $\phi^{-1}$ by $\gamma^{-1}$ or by $-\phi$ in the exact sequence (\ref{yonedarep}).

By the general result \cite[Lem. 4.7]{omac} there exist automorphisms $\kappa^1$ and $\kappa^2$ of $X$ with the property that the (exact) diagram
\begin{equation}\label{omaclemma}
\begin{CD} 0 @>  >> A^t(F)_p @> \iota\circ\gamma^{-1} >> X @> \Theta >> X @> \pi >> A(F)^*_p @>  >> 0\\
@. @\vert @V \kappa^1 VV  @V \kappa^2 VV   @\vert @. \\
0 @>  >> A^t(F)_p @> \iota\circ(-\phi) >> X @> \Theta >> X @> \pi >> A(F)^*_p @>  >> 0
\end{CD}
\end{equation}
is commutative.

Now, one may compute (for instance, by the argument of \cite[Prop. 4.4]{bleymc}) the desired sum as
$$\sum_{\psi\in\widehat G}\frac{\varepsilon_\psi(\Psi)}{\varepsilon_\psi(\Psi')}\cdot e_\psi=\frac{{\rm det}_{\QQ_p[G]}(\langle\gamma,\Theta,s_1,s_2\rangle)}{{\rm det}_{\QQ_p[G]}(\langle-\phi^{-1},\Theta,s_1,s_2\rangle)}$$
where, for any ($\QQ_p[G]$-equivariant) splittings $s_1$ and $s_2$ as in (\ref{splitting1}) and (\ref{splitting2}), and any endomorphism $\beta$ of $A^t(F)_p$, we have written $\langle\beta,\Theta,s_1,s_2\rangle$ for the composite isomorphism
\begin{multline*}\QQ_p\cdot X\stackrel{s_1}{\to}\QQ_p\cdot A^t(F)_p\oplus\QQ_p\cdot\im(\Theta)\stackrel{\beta\oplus{\rm id}}{\to}\QQ_p\cdot A^t(F)_p\oplus\QQ_p\cdot\im(\Theta)\\ \to\QQ_p\cdot A(F)_p^*\oplus\QQ_p\cdot\im(\Theta)\stackrel{s_2^{-1}}{\to}\QQ_p\cdot X,\end{multline*}
with the unlabeled arrow simply mapping a point $P^t_{(s,i)}$ to $P^*_{(s,i)}$.

It is then straightforward to deduce from the commutativity of the diagram (\ref{omaclemma}) that 
$$\sum_{\psi\in\widehat G}\frac{\varepsilon_\psi(\Psi)}{\varepsilon_\psi(\Psi')}\cdot e_\psi={\rm det}_{\ZZ_p[G]}(\kappa_1)\cdot{\rm det}_{\ZZ_p[G]}(\kappa_2)^{-1}$$
and, since both of the determinants on the right hand side are by construction elements of $\ZZ_p[G]^\times$, this concludes the proof of Lemma \ref{still to do}.
\end{proof}


\section{Computation of the Mazur-Tate pairing}\label{comp mtp}

In this section we explain how one may numerically compute the Mazur-Tate pairing
(\ref{mtpairing}). The computation can be reduced to the computation of local Tate duality pairings which, in turn, may in simple situations be
computed by the evaluation of Hilbert symbols thanks to recent results of Fisher and Newton \cite{FisherNewton} or of Visse \cite{Visse}.

Our approach is based crucially on the ability to compute certain generalised Selmer groups,
for whose calculation we will apply a  method of Schaefer and Stoll \cite{SS03}. In this regard we also wish to mention subsequent work of
  Cremona, Fisher, O'Neil, Simon and Stoll \cite{Cremona_et_al_I,Cremona_et_al_II,  Cremona_et_al_III}
  where they develop algorithms for computing $n$-Selmer groups by representing their elements as curves of
degree $n$ in $\mathbb{P}^{n-1}$. However, we have not so far required using their methods.

\subsection{The general strategy}\label{strategy}

We continue to assume the hypotheses of Section \ref{hyp}.
In particular,as explained in Section \ref{mtregs}, the result \cite[Prop. 6.3(ii)]{rbsd} implies that every element of $A^t(k)_p$ and $A(k)_p$ is `locally-normed'.
Under this condition, Bertolini and Darmon have defined in \cite[\S 3.4.1]{bert} and \cite[\S 2.2]{bert2} a pairing
\[
\langle\ , \ \rangle_1 \colon A^t(k)_p \otimes_{\ZZ_p} A(k)_p \lra G \simeq \IpG / \IpG^2.
\] 
Although the definition of this pairing is only given in the case that $A$ is an elliptic curve, 
it extends
naturally to our more general setting (see also  \cite[\S 10]{rbsd}). 


The results of Bertolini and Darmon in \cite[Thm.~2.8 and Rem.~2.10]{bert2} and of Tan in \cite[Prop.~3.1]{kst} combine to directly show
that the pairing $ \langle\ , \ \rangle_1$ coincides with the 
Mazur-Tate pairing $\langle\ , \ \rangle_{F/k}^{\rm MT}$. 
We are therefore left with the task to describe the explicit computation of $\langle\ , \ \rangle_1$.

Let $B$ be either $A$ or its dual $A^t$. For a finite set $S$ of non-archimedean places of $k$ we define the
generalised Selmer group
\[
\Sel_S^{(p^n)}(B/F) \le H^1(F, B[p^n])
\]
to be the kernel of the localisation map
\[
H^1(F, B[p^n]) \lra \prod_{w \not\in S(F)} H^1(F_w, B),
\]
with the product running over all non-archimedean places of $F$ that do not belong to the set $S(F)$ of places that lie above a place in $S$. We recall that this group is also often referred to as a `relaxed Selmer group'.

By Kummer theory we then have
\begin{eqnarray*}
&& \Sel_S^{(p^n)}(B/F)  \\
&=& \{ \xi \in H^1(F, B[p^n]) \mid \res_w(\xi) \in \delta_w( B(F_w) / p^nB(F_w)) \text{ for all } w \not\in S(F) \}.
\end{eqnarray*}
Here $\res_w:H^1(F, B[p^n])\to H^1(F_w,B[p^n])$ denotes the canonical localisation map and
$\delta_w:B(F_w) / p^nB(F_w)\to H^1(F_w,B[p^n])$ is the canonical Kummer map.
In particular, when $S$ is taken to be the empty set, one recovers the usual Selmer group $\Sel^{(p^n)}(B/F)$ associated with multiplication by $p^n$.

In the following we will employ the notation from \cite[Sec.~10.2.1]{rbsd}.
We set $Z := \Ze/p^n\Ze$ and  $R := Z [G]$ and define additional $R$-modules
\[
B_S(\mathbb{A}_F) / p^n := \prod_{w \in S(F)} B(F_w)/p^n B(F_w)
\]
and
\[
H_S^1( \mathbb{A}_F, B[p^n]) := \prod_{w \in S(F)} H^1(F_w, B[p^n]) .
\]

{In order to define $\langle\ , \ \rangle_1$, we first recall the construction of a canonical (local duality) perfect pairing 
\begin{equation}\label{Rpairing}
\langle\ , \ \rangle \colon H_S^1( \mathbb{A}_F, A^t[p^n]) \times H_S^1( \mathbb{A}_F, A[p^n]) \lra R.
\end{equation}

If $A$ is defined over an $\ell$-adic field $L$ (for some prime $\ell$), then we write $\langle \ , \  \rangle_{L, p^n}$ for
the local Tate duality pairing obtained by combining the cup product, the Weil pairing and the invariant map as follows:
\begin{eqnarray}\label{local Tate pairing}
   H^1(L, A^t[p^n]) \times H^1(L, A[p^n]) &\stackrel{\cup} \lra& H^2(L, A^t[p^n]  \otimes_\Zp A[p^n]) \\
\notag &\lra& H^2(L, \mu_{p^n}) \\
\notag &\stackrel{\rm{inv_L}}\lra& \Qp/\Zp.
\end{eqnarray}
Then, for $x=(x_w)_{w \in S(F)} \in H_S^1( \mathbb{A}_F, A^t[p^n])$ and $y=(y_w)_{w\in S(F)} \in H_S^1( \mathbb{A}_F, A[p^n])$ we set
\[
\langle x, y \rangle_{S} := \sum_{w \in S(F)} \langle x_w, y_w \rangle_{F_w, p^n}.
\]
All values $\langle x, y \rangle_{S}$ in fact belong to $\frac{1}{p^n} \Zp/\Zp$, which we henceforth identify with
$Z = \Ze / p^n\Ze$.
We may thus define the pairing (\ref{Rpairing}) by the explicit formula
\begin{equation}\label{equiv BD pairing}
\langle x, y \rangle := \sum_{g \in G} \langle x^g, y \rangle_{S} g^{-1}.
\end{equation}
}

We now recall the explicit definition of $\langle P, Q \rangle_1$ for $P \in A^t(k)$ and $Q \in A(k)$. Let $\Sigma$ be an admissible
set of primes as in \cite[Def.~2.22]{bert}, \cite[Def.~1.5]{bert2} or \cite[Lem.~10.5]{rbsd}.
A crucial consequence of the definition of admissibility is that the canonical (diagonal) localisation map
  \begin{equation}\label{injectivity}
    A(F)/p^nA(F) \lra A_\Sigma(\mathbb{A}_F) / p^n
  \end{equation}
is injective.


We write $\delta$ for the canonical global Kummer map and let
$\tilde x \in \Sel_\Sigma^{(p^n)}(A^t/F)^G$ denote the image of $P$ under the canonical composition
\[
A^t(k)  \lra (A^t(F)/p^nA^t(F))^G \stackrel{\delta^G}{\lra} \Sel_\Sigma^{(p^n)}(A^t/F)^G.
\]
We also let $\tilde y \in (A_\Sigma(\mathbb{A}_F) / p^n)^G $ be the image of $Q$ under
the canonical (diagonal) localisation map
\[
A(k) \lra (A(F)/p^nA(F))^G\lra (A_\Sigma(\mathbb{A}_F) / p^n)^G.
\]

By \cite[\S 3.1]{bert} the $R$-modules $\Sel_\Sigma^{(p^n)}(A^t/F)$ and $A_\Sigma(\mathbb{A}_F) / p^n$
are $G$-cohomologically trivial. We therefore find
elements $x \in \Sel_\Sigma^{(p^n)}(A^t/F)$ and $y=(y_w)_{w \in \Sigma(F)} \in A_\Sigma(\mathbb{A}_F) / p^n$ such that
\[
\Tr_G(x) = \tilde x, \quad \Tr_G(y) = \tilde y.
\]

We next consider the canonical (diagonal) localisation map
\[
  \lambda_\Sigma \colon \Sel_\Sigma^{(p^n)}(A^t/F) \sseq H^1(F, A^t[p^n]) \stackrel{\oplus \res_w} \lra
 H_\Sigma^1( \mathbb{A}_F, A^t[p^n])
\]
and the product of local Kummer maps
\[
  \delta_\Sigma \colon A_\Sigma(\mathbb{A}_F) / p^n \stackrel{\oplus \delta_w}\lra
  H_\Sigma^1( \mathbb{A}_F, A[p^n]).
\]


Then, by the definition of $\langle \ , \ \rangle_1$  given in \cite[\S 2.2]{bert2}, we obtain
\[
\langle P, Q \rangle_1 \equiv \langle \lambda_\Sigma(x), \delta_\Sigma(y) \rangle \pmod{I_p(G)^2},
\]
with $\langle\ , \ \rangle$ as in (\ref{Rpairing}).
Noting the sign involved in the definition (\ref{equivariantreg}) of the equivariant regulator,
we are in fact interested in computing the inverse $-\langle P, Q \rangle_1$.
From the definition (\ref{equiv BD pairing}) of $\langle\ , \ \rangle$, we easily derive the explicit expression 
\begin{equation}\label{explicit formula 2}
  -\langle P, Q \rangle_1 \equiv \sum_{g \in G} \left( \sum_{w \in \Sigma(F)}
    \langle \res_w(x^g), \delta_w(y_w) \rangle_{F_w,p^n} \right) g^{}
 \pmod{I_p(G)^2},
\end{equation}
which is convenient for the explicit evaluations which we will describe in the next subsection.

\subsection{Algorithmic evaluation of the pairing}
In this subsection we describe how we numerically evaluate the right hand side of (\ref{explicit formula 2}). From an algorithmic point of view
we are mainly interested in the case of elliptic curves and, for this reason, we henceforth assume that $E := A = A^t$
is an elliptic curve defined over $k$. Furthermore, we assume that $F/k$ is cyclic of degree $p$ where as before
$p$ is an odd prime such that $E$, $F/k$  and $p$ satisfy  the hypotheses of Section \ref{hyp}.

The central computational problem in the numerical evaluation of the Mazur-Tate pairing, using the
  approach of Bertolini and Darmon as described in the previous section, is first of all the computation of the generalised Selmer group
  $\Sel^{(p)}_\Sigma(E/F)$. We will closely follow the method of Schaefer and Stoll \cite{SS03} to compute this group.

We fix a finite set $V$ of places of $F$ containing the $p$-adic places and all places $w$ such that the Tamagawa number
$c_w$ of $E$ at $w$ is divisible by $p$. We also fix an admisible set $\Sigma$ of places of $k$ as in Section \ref{strategy}.
We then write $H^1(F, E[p]; V\cup \Sigma(F))$ for the group of cohomology classes in $H^1(F, E[p])$ that are unramified outside
$V \cup \Sigma(F) $. Now
the result of \cite[Prop.~3.2]{SS03} shows that $\Sel^{(p)}(E/F)$ is  given by 
\[
   \{ \xi \in H^1(F, E[p]; V \cup \Sigma(F))
      \mid \res_w(\xi) \in \delta_w(E(F_w)/pE(F_w)) \text{ for all } w \in V \cup \Sigma(F) \},
\]
whereas a slight generalization of the above mentioned result implies that the generalised Selmer group $\Sel_\Sigma^{(p)}(E/F)$
is equal to
\begin{equation}\label{localconditions}
  \{ \xi \in H^1(F, E[p]; V \cup \Sigma(F)) \mid
  \res_w(\xi) \in \delta_w(E(F_w)/pE(F_w)) \text{ for all } w \in V \setminus \Sigma(F) \}.
\end{equation}

We fix an algebraic closure $F^c$ of $F$ and let $W \sseq E[p] \setminus \{0\}$ be a $G_F$-invariant
spanning set for $E[p]$. Thus $W$ is a union of $G_F$-orbits containing an $\Fp$-basis of $E[p]$.
  Note that in the generic case
  $G_F$ acts transitively on $E[p] \setminus \{0\}$, so that in such cases we are forced to use $W = E[p] \setminus \{0\}$.

We write $A^c := \Map(W, F^c)$ for the
set of maps from $W$ to $F^c$. Then the Galois group $G_F$ acts on $A^c$ by conjugation, i.e., for $\sigma \in G_F$,
  $a \in A^c$ and $P \in W$ one has $(\sigma a)(P) = \sigma( a(\sigma^{-1}(P)) )$.

We denote by $A = \Map_{G_F}(W, F^c) := \Map(W, F^c)^{G_F}$ the set of $G_F$-invariant maps in $A^c$.
Then $A$ is a finite dimensional {\'e}tale $F$-algebra.
Explicitly, let $P_1, \ldots, P_s \in W$ be a set of $G_F$-orbit representatives of $W$ and set
  \[
    H_i := \{ \sigma \in G_F \mid \sigma(P_i) = P_i \}
  \]
and $L_i := {(F^c)}^{H_i}$. Then each $L_i/F$ is a finite separable field extension
and we have a canonical isomorphism
\[
A \lra \prod_{i=1}^s L_i, \quad a \mapsto \left( a(P_i) \right)_{1 \le i \le s}.
\]

The Weil pairing $e_p \colon E[p] \times E[p] \lra \mu_p(F^c)$ defines a map
\[
\omega \colon E[p] \lra  \mu_p(A^c) := \Map(W, \mu_p(F^c)), \quad P \mapsto e_p(P, \ \underline\ \ ).
\]
This map induces a homomorphism in cohomology
\begin{equation}\label{inducedbyWeil}
\bar\omega \colon H^1(F, E[p]) \lra H^1(F, \mu_p(A^c))
\end{equation}
which, by \cite[Prop.~4.3]{SS03}, is known to be injective if $p \nmid |\Gal(F(E[p])/F)|$ or $p \nmid |W|$.

In addition, by an immediate generalization of Hilbert's Theorem 90, we have a Kummer isomorphism 
\begin{equation}\label{extendedKummer}
  \kappa \colon  H^1(F, \mu_p(A^c)) \lra A^\times / A^{\times p}
\end{equation}
which combined with $\bar\omega$ defines an embedding (assuming, for example, that $p \nmid |W|$)
\[
H^1(F, E[p]) \hookrightarrow A^\times / A^{\times p}.
\]

Roughly speaking, the algorithm of Schaefer and Stoll \cite{SS03}, in a first step, computes
$H^1(F, E[p]; V \cup \Sigma(F)) $ as a subset
of $A^\times / A^{\times p}$ via the embedding $\kappa \circ \bar\omega$ and then, in a second step, checks the local conditions
occurring in (\ref{localconditions}).

To describe the first step we let $L/F$ be a finite extension. For a finite place $t$ of $L$ we write
  $\ord_t \colon L^\times \lra \Ze$ for the associated normalised valuation.
If $T$ is a finite set of finite places of $F$, we set
\[
L(T,p) := \{ a \in L^\times/L^{\times p} \mid \ord_t(a) \in p\Ze \text{ for all } t \not\in T(L) \},
\]
and more generally, if $A \simeq \prod_{i=1}^s L_i$ is an \'etale $F$-algebra as above, we define 
\[
A(T, p) := \prod_{i=1}^s L_i(T, p).
\]
Then, using the embedding $H^1(F, E[p]) \hookrightarrow A^\times / A^{\times p}$ given by $\kappa\circ\bar\omega$,
the result of \cite[Cor.~5.9]{SS03} shows the equality
\[
H^1(F, E[p]; V \cup \Sigma(F))  = H^1(F, E[p]) \cap A(V \cup \Sigma(F), p).
\]

The work of Schaefer and Stoll now describes how to compute the group $H^1(F, E[p]) $ inside $A^\times / A^{\times p}$
and then, by intersecting with the finite group $A(V \cup \Sigma(F), p)$, we obtain $H^1(F, E[p]; V \cup \Sigma(F))$ as a subgroup of 
$A^\times / A^{\times p}$. For further details we refer the reader to \cite{SS03}.

Testing the local conditions in the second step (see again \cite{SS03} for details)
we finally obtain both $\Sel^{(p)}(E/F)$ and $\Sel_\Sigma^{(p)}(E/F)$
as subgroups of $A^\times / A^{\times p}$, or even better, as subgroups of the  finite group $A(V \cup \Sigma(F), p)$.
Note, however, that the computation of  $A(V \cup \Sigma(F), p)$ requires the computation of ideal class groups and units
in all of the fields $L_i$, $i = 1,..,s$.

In this context we recall that in the generic case we have to use $W = E[p] \setminus \{0\}$. In any such cases
  we will fix $S_0 \in  W$ and set $L := F(S_0)$, so that $[L:F] = p^2-1$ and $A \simeq L$.

For any place $w$ in $\Sigma(F)$ we fix an embedding $\iota_w:F\to F_w$ and
set $A_w^c :=  F_w^c \otimes_F A$ and $A_w := F_w \otimes_F A$. Note that there is a canonical isomorphism
\[
A_w  \stackrel{\simeq}\lra \Map_{G_{F_w}}(W, F_w^c), \quad z \otimes a \mapsto \left( S \mapsto z\iota_w(a(S)) \right).
\]
We then consider the following commutative diagram
\begin{equation}\label{F and F_w diagram}
\xymatrix{
  E(F)/pE(F) \ar[r]^\delta \ar[d]^{\iota_w} & H^1(F, E[p]) \ar[r]^{\bar\omega} \ar[d]^{\res_w} & H^1(F, \mu_p(A^c)) \ar[r]^\kappa \ar[d]^{\res_w} &
  A^\times / A^{\times p} \ar[d]^{\iota_{w,*}} \\
  E(F_w)/pE(F_w) \ar[r]^{\delta_w} & H^1(F_w, E[p]) \ar[r]^{\bar\omega_w} & H^1(F_w, \mu_p(A_w^c)) \ar[r]^{\kappa_w}  &
  A_w^\times / A_w^{\times p},
}
\end{equation}
where $\iota_{w,*}$ is induced by composition with $\iota_w$ while the maps $\bar\omega_w$ and $\kappa_w$ are the local analogues of
  (\ref{inducedbyWeil}) and (\ref{extendedKummer}), respectively, and $\res_w$ is the localisation map.

The local Tate pairing 
\[
\langle\ , \ \rangle_{F_w,p} \colon H^1(F_w, E[p]) \times H^1(F_w, E[p])  \lra \Qp / \Zp
\]
of (\ref{local Tate pairing}) induces a pairing $\langle\ , \ \rangle_{A_w}$ on the image of $\kappa_w \circ \bar\omega_w$. In what follows we will describe
the computation of  $\langle\ , \ \rangle_{A_w}$.

We write $q_{A_w}$ for the unique quadratic form such that, for all 
\[
a,b \in (\kappa_w \circ \bar\omega_w)\left( H^1(F_w, E[p])  \right),
\]
one has
\[
\langle a, b \rangle_{A_w} = q_{A_w}(ab) - q_{A_w}(a) - q_{A_w}(b).
\]

Let
\[
  \{\ , \ \}_{F_w,p} \colon F_w^\times/F_w^{\times p} \times  F_w^\times/F_w^{\times p} \lra \mu_p
\]
denote the Hilbert symbol.
We henceforth assume that $E[p]$ is contained in $E(F_w)$ (which is always the case for places $w \in \Sigma(F)$ for an admissible set
$\Sigma$ of places of $k$). In this particular case one can shift the problem of computing the local Tate pairing to the computation
of Hilbert symbols which is easy for places $w$ which are prime to $p$ by \cite[Ch.~V, Prop.~3.4]{Neu}.

We fix generators $S, T \in E(F_w)$ of $E[p]$. Then $e_p(T, S)$ generates $\mu_p$  and we define a map
\[
  \xi_{S,T} \colon \mu_p \lra \Ze / p\Ze
\]
by $\xi_{S,T}(e_p(T, S)) =  1 + p\Ze$.  For $\bar a \in A_w^\times / A_w^{\times p}$ with $a \in A_w \simeq \Map_{G_w}(W, \bar F_w)$ we obtain a well defined map
  $\bar a \colon W \lra F_w^\times / F_w^{\times p}$.

Fisher and Newton in \cite{FisherNewton} have given an explicit description of the local Tate pairing
  associated with the $p$-torsion of an elliptic curve in the special case $p=3$. Their formulas have subsequently been generalised by Visse
  in  \cite{Visse} to any odd prime $p$. In particular, from \cite[Th.~3.7 and 3.11]{Visse} we obtain, for arbitrary odd $p$, an equality
\[
q_{A_w}(\bar a) = \xi_{S,T} \left( \left\{ \bar a(S), \bar a(T)  \right\}_{F_w,p} \right)
\]
for all $\bar a \in  (\kappa_w \circ \bar\omega_w)\left( H^1(F_w, E[p])  \right)$. As a consequence we obtain
\begin{equation}\label{local Tate formula}
  \langle \bar a, \bar b \rangle_{A_w} = \xi_{S,T}
  \left( \frac{\{\bar a(S), \bar b(T)\}_{F_w,p}}{\{\bar a(T), \bar b(S)\}_{F_w,p}} \right).
\end{equation}

We are now in a position to describe the explicit computation of $-\langle P, Q \rangle_1$
using the formula in (\ref{explicit formula 2}).
Let $\tilde y$ denote the image of $Q$ in $E_\Sigma(\mathbb{A}_F)/p$ as in Section \ref{strategy}.
Assuming that the point $Q$ is not divisible by $p$ ensures that $\tilde y \ne 0$ because of the injectivity of the map (\ref{injectivity}).

For each $v \in \Sigma$ we fix a place $\hat v$ of $F$ lying over $v$. For each $w\in\Sigma(F)$ we define
\[
y_w := \begin{cases} \tilde y_{\hat v}, & \text{if } w \mid v \text{ and } w = \hat v, \\
                                   0,                         & \text{ otherwise }.
            \end{cases}
\]
Since $\tilde y$ is fixed by $G$ it is easy to see that $y := (y_w)_{w \in \Sigma(F)} \in E_\Sigma(\mathbb{A}_F)/p$
satisfies $\Tr_G(y) = \tilde y$. Hence the formula in (\ref{explicit formula 2}) simplifies to
 \begin{equation}\label{explicit formula 3}
-\langle P, Q \rangle_1 \equiv \sum_{g \in G} 
\left( \sum_{v \in \Sigma} \langle \res_{\hat v}(x^g), \delta_{\hat v}(y_{\hat v}) \rangle_{F_{\hat v},p} \right) g^{}
\pmod{I_p(G)^2}.
\end{equation} 
    
Next we explain how to compute $x \in \Sel^{(p)}_\Sigma(E/F)$ with $\Tr_G(x) = \tilde x$. The generalised Selmer group
$\Sel^{(p)}_\Sigma(E/F)$ is computed as a subgroup of $A^\times / A^{\times p}$ which carries an action of $G$ induced by conjugation.
Explicitly, for all $\sigma \in G$, $a \in A$ and $S \in W$ we choose an arbitrary lift $\hat\sigma \in G_k$ of $\sigma$ and set
\begin{equation}\label{G action}
(\sigma a)(S) = \hat\sigma(a(\hat\sigma^{-1}(S))).
\end{equation}
As sketched in Section \ref{comm impl} below this action is explicitly computable. We can therefore compute the $\Fp$-representation
of $G$ induced by $\Sel^{(p)}_\Sigma(E/F)$ and then represent $\Tr_G$ as a linear map 
$\Sel^{(p)}_\Sigma(E/F) \lra \Sel^{(p)}_\Sigma(E/F)$ and thus compute a preimage $x$ of $\tilde x$. 

In the same way we can compute the elements $x^g$ occuring in (\ref{explicit formula 3}).
Note that as an element of 
$A^\times / A^{\times p}$ the element $\tilde x$ is given by $(\kappa\circ\bar\omega\circ\delta)(P)$. The computation of the map
  \[
    H := \kappa \circ \bar\omega \circ \delta
  \]
  is quite non-trivial and explained in \cite[Ch.~2]{Sch15a, Sch15b} and \cite[Ch.~3]{SS03},
  as well as the computation of the local analogue
  \[
    H_w := \kappa_w \circ \bar\omega_w \circ \delta_w.
  \]

For each $v \in \Sigma$ we now compute the completion $F_{\hat v}$ together with an embedding 
$\iota_v \colon F \lra F_{\hat v}$ and define $a_{P, g, \hat v} := \res_{\hat v}(x^g) = \iota_v \circ x^g$.

In order to compute $\delta_{\hat v}(y_{\hat v})$ we use the commutativity of diagram  (\ref{F and F_w diagram}) and define
\[
a_{Q, \hat v} := \kappa_{\hat v}( \bar\omega_{\hat v}( \delta_{\hat v}( \iota_{\hat v}(Q) ))) =
\iota_{\hat v} \circ ( \kappa( \bar\omega(\delta(Q))) = 
\iota_{\hat v} \circ H(Q)
\]
and finally obtain
\[
\langle \res_{\hat v}(x^g), \delta_{\hat v}(y_{\hat v}) \rangle_{F_{\hat v},p} =
\xi_{S,T}\left(
   \frac{\{a_{P,g, \hat v}(S), a_{Q, \hat v}(T)\}_{F_w,p}}{\{a_{P,g,\hat v}(T), a_{Q, \hat v}(S)\}_{F_w,p}} 
\right).
\]

\subsection{Comments on the implementation}\label{comm impl}
In this subsection we discuss the restrictions on our MAGMA implementation, where we must always assume 
  that $G_F$ acts   transitively on  $E[p] \setminus \{0\}$. Hence, for the rest of this subsection, we set
  $W = E[p] \setminus \{0\}$ and fix a point $S_0 \in W$. As before, $F/k$ is a cyclic extension of degree $p$ where $p$ is an odd
prime.

We choose $\lambda \in \Qu$ such that the elements
\[
w_S := y_S + \lambda x_S, \quad S = (x_S, y_S) \in W,
\]
are pairwise distinct. Since $G_F$ acts transitively on $W$ the polynomial
\[
f(x) := \prod_{S \in W} (x - w_S) \in F[x]
\]
is irreducible. It is easily seen that $F(S_0) = F(w_{S_0})$. We therefore may and will use $L := F(w_{S_0})$.

This has the following
useful consequence for the computation of $a_w := \iota_w \circ a$, where  $a$ is an element of $A = \Map_{G_F}(W, F^c)$ and
$\iota_w$ an embedding of $F$ into $F_w$. Namely, if we assume that $a \in A$ corresponds to
$\bar h \in F[x] / (f(x))$ under the composite map
\[
A \stackrel\alpha\lra L \stackrel{\beta^{-1}}\lra F[x]/(f(x)), \quad \alpha(a) := a(S_0), \quad \beta(\bar h) := h(w_{S_0}),
\]
then we have $a(S_0) = h(w_{S_0})$. 
Assume further that $\tilde S \in E(F_w)[p]$ and let $S \in W$ be such that $\iota_w(S) = \tilde S$.
Then a straightforward computation shows
that $a_w(\tilde S) = (\iota_wh)(w_{\tilde S})$.

A further useful consequence is a particularly simple description of the action of $G$ on $A$. Recall that $\sigma \in G$ acts on $A$
by conjugation as in (\ref{G action}). If
$\bar h \in  F[x]/(f(x))$ corresponds to $a \in A$, then we claim that $\sigma a$ corresponds to $\overline{\sigma h}$.
Indeed, if $\tau(S_0) = \hat\sigma^{-1}(S_0)$ with $\tau \in G_F$, then
\[
(\sigma a)(S_0) = \hat\sigma(a(\hat\sigma^{-1}(S_0))) = \hat\sigma(\tau(h(w_{S_0}))) = \hat\sigma(h(\tau(w_{S_0}))) = (\sigma h) (w_{S_0}).
\]

In our implementation we additionally assume $k=\Qu$ and $p=3$. In this case  explicit formulae for the computation
of the generalised Selmer group based on the work of Schaefer and Stoll in
\cite{SS03} are given in   \cite[Sec.~7]{Geis}. These are used throughout in our implementation.

We summarise the above discussion in the following remark.
\begin{remark}{\em
  Our MAGMA implementation is restricted to handle the case $k=\Qu$, $p=3$ and $\mathrm{rk}(E(F)) = \mathrm{rk}(E(\Qu)) = r$ with $r > 0$.
  Assuming the validity of the classical Birch and Swinnerton-Dyer conjecture for $E/F$ the implementation will check the
  hypotheses (a) - (i) of Section \ref{hyp} as well as triviality of $\Sha_p(E_F)$
  and then numerically verify the rationality part of the
  eTNC up to the precision of the computation.
  We then rigorously prove the condition from Theorem \ref{main} (by rewriting it in terms of congruence relations in the manner suggested in Remark \ref{acta}). 
}\end{remark}

\subsection{Computations}\label{computational results}

Our MAGMA implementation can be used to numerically verify the validity of the refined Birch and Swinnerton-Dyer conjecture for pairs
$(E,F/k)$ where
\begin{itemize}
  \item $k=\Qu$ and $E/\Qu$ is an elliptic curve with global minimal Weierstrass equation,
  \item $F$ is the unique subfield of degree $p=3$ in $\Qu(\zeta_\ell)/\Qu$ for primes $\ell \equiv 1 \pmod{3}$,
  \item $\mathrm{rk}(E(F)) = \mathrm{rk}(E(\Qu)) = r$ for some $r > 0$,
  \item $E$ and $F$ satisfy the hypotheses (a) - (i) of Section \ref{hyp}.
\end{itemize}
Note that the condition   $\mathrm{rk}(E(F)) = \mathrm{rk}(E(\Qu)) = r$ ensures that $E(F)_p \simeq \Ze_p^r$
is not projective as a $\ZpG$-module.

Here is the list of examples for $r=1$, $\ell < 50$ and conductor less than $100$. We specify elliptic curves by their Cremona reference. 
\begin{itemize}
\item [-]  $E$ is $37a1$ and $\ell \in \{13,19\}$.
\item [-]  $E$ is $43a1$ and $\ell \in \{7,13,37\}$.
\item [-]  $E$ is $53a1$ and $\ell \in \{13, 19, 31,43 \}$.
\item [-]  $E$ is $58a1$ and $\ell \in \{7,13,19,31,43 \}$.
\item [-]  $E$ is $61a1$ and $\ell \in \{7, 13, 43 \}$.
\item [-]  $E$ is $65a2$ and $\ell \in \{19, 37, 43 \}$.
\item [-]  $E$ is $77a1$ and $\ell \in \{19, 37 \}$.
\item [-]  $E$ is $79a1$ and $\ell \in \{13, 19, 37 \}$.
\item [-]  $E$ is $82a1$ and $\ell \in \{13, 19, 43 \}$.
\item [-]  $E$ is $82a2$ and $\ell \in \{13, 19, 43 \}$.
\item [-]  $E$ is $83a1$ and $\ell \in \{7, 13, 37 \}$.
\item [-]  $E$ is $88a1$ and $\ell \in \{7,43 \}$.
\item [-]  $E$ is $89a1$ and $\ell \in \{19, 31, 37 \}$.
\item [-]  $E$ is $91a1$ and $\ell \in \{31 \}$.
\item [-]  $E$ is $389a1$ and $\ell \in \{7,13,43 \}$.
\item [-]  $E$ is $433a1$ and $\ell \in \{7, 13, 31, 37 \}$.
\item [-]  $E$ is $446d1$ and $\ell \in \{19 \}$.
\end{itemize}

For $r=2$, $\ell < 50$ and conductor less than $500$ we computed the following examples.
\begin{itemize}
\item [-]  $E$ is $389a1$ and $\ell \in \{7,13,43 \}$.
\item [-]  $E$ is $433a1$ and $\ell \in \{7, 13, 31, 37 \}$.
\item [-]  $E$ is $446d1$ and $\ell \in \{19 \}$.
\end{itemize}

\subsection*{Acknowledgements} 
The authors are very grateful to David Burns for many interesting discussions and much encouragement, as well as to Christian Wuthrich for his constant interest in our work and many related discussions. We also wish to thank the referee for their careful reading of the manuscript. In particular,
 their comments and suggestions on Section \ref{comp mtp} significantly helped to improve the presentation of this part of the manuscript.

The first author wants to thank Chris Geishauser who in the context of his master thesis \cite{Geis}  provided
many MAGMA routines used in the computation of our numerical examples.

The second author is grateful to Stefano Vigni for some pertinent discussions. He also acknowledges financial support from the 
Spanish Ministry of Science and Innovation, through the `Severo Ochoa Programme for Centres of Excellence in R\&D' [SEV-2015-0554] and [CEX-2019-000904-S] as well as through projects [MTM2016-79400-P] and [PID2019-108936GB-C21].

\Addresses

\begin{thebibliography}{10}




\bibitem{bert} M. Bertolini, H. Darmon,
\newblock Derived heights and generalized Mazur-Tate regulators,
\newblock Duke Math Journal {\bf 76} No. 1 (1994) pp. 75-111.

\bibitem{bert2} M. Bertolini, H. Darmon,
\newblock Derived $p$-adic heights,
\newblock American Journal of Math {\bf 117} (1995) pp. 1517-1554.

\bibitem{bleyone} W. Bley,
\newblock Numerical evidence for the equivariant Birch and
Swinnerton-Dyer conjecture,
\newblock Exp. Math. \textbf{20} (2011), 426-456.

\bibitem{bleytwo} W. Bley,
\newblock Numerical evidence for the equivariant Birch and
Swinnerton-Dyer conjecture (part II),
\newblock Math. Comp. \textbf{81} (2012), 1681-1705.

\bibitem{bleythree} W. Bley,
\newblock The equivariant Tamagawa number conjecture and modular symbols,
\newblock Math. Ann. \textbf{356} (2013), 179-190.



\bibitem{bleymc} W. Bley, D. Macias Castillo,
\newblock Congruences for critical values of higher derivatives of twisted Hasse-Weil $L$-functions,
\newblock J. reine u. angew. Math. {\bf 722} (2017) 105-136.










\bibitem{bufl99} D. Burns, M. Flach,
\newblock Tamagawa numbers for motives with (non-commutative) coefficients,
\newblock Doc. Math. {\bf 6} (2001) 501-570.









\bibitem{gm} D. Burns, M. Kurihara, T. Sano,
\newblock On zeta elements for $\mathbb{G}_m$,
\newblock Documenta Math. \textbf{21} (2016) 555-626.




\bibitem{omac} D. Burns, D. Macias Castillo,
\newblock Organising matrices for arithmetic complexes,
\newblock Int. Math. Res. Notices {\bf 2014} 10 (2014) 2814-2883.

\bibitem{rbsd} D. Burns, D. Macias Castillo,
\newblock On refined conjectures of Birch and Swinnerton-Dyer type for Artin-Hasse-Weil $L$-series,
\newblock submitted for publication.


\bibitem{bmw} D. Burns, D. Macias Castillo, C. Wuthrich,
\newblock On Mordell-Weil groups and congruences between derivatives of twisted Hasse-Weil $L$-functions,
\newblock J. reine angew. Math. {\bf 734} (2017) 187-228.









\bibitem{BV2} D. Burns, O. Venjakob,
\newblock On descent theory and main conjectures in non-commutative
Iwasawa theory,
\newblock J. Inst. Math. Jussieu {\bf 10} (2011) 59-118.


\bibitem{Cremona_et_al_I} J.E. Cremona, T.A.Fisher, C. O'Neil, D. Simon, M. Stoll,
\newblock Explicit $n$-decent on elliptic curves, I. Algebra,
\newblock J.reine angew. Math {\bf 615} (2008) 121-155.

\bibitem{Cremona_et_al_II} J.E. Cremona, T.A.Fisher, C. O'Neil, D. Simon, M. Stoll,
\newblock Explicit $n$-decent on elliptic curves, I. Geometry,
\newblock J.reine angew. Math {\bf 632} (2009) 63-84.

\bibitem{Cremona_et_al_III} J.E. Cremona, T.A.Fisher, C. O'Neil, D. Simon, M. Stoll,
\newblock Explicit $n$-decent on elliptic curves, I. Algorithms,
\newblock Mathematics of Computation {\bf 84}, Vol.292, (2014) 895-922.

\bibitem{cfksv} J. Coates, T. Fukaya, K. Kato, R. Sujatha, O. Venjakob,
\newblock The ${\rm GL}_2$-main conjecture for elliptic curves without complex multiplication,
\newblock Publ. IHES {\bf 101} (2005) 163-208.



\bibitem{darmon} H. Darmon,
\newblock A refined conjecture of Mazur-Tate type for Heegner points,
\newblock Invent. Math. {\bf 110} (1992) 123-146.




\bibitem{FisherNewton} T. Fisher, R. Newton,
\newblock Computing the Cassels-Tate pairing on the $3$-Selmer group of an elliptic curve,
\newblock Int. J. Number Theory {\bf 10} (2014), no. 7, 1881-1907.






\bibitem{Geis} C. Geishauser,
\newblock Computation of $2$-extensions of dual Selmer groups,
\newblock Master thesis, LMU 2018.






\bibitem{HS} P. J. Hilton, U. Stammbach,
\newblock A course in Homological Algebra,
\newblock Springer-Verlag, New York, 1970.











\bibitem{lw} T. Lawson, C. Wuthrich,
\newblock Vanishing of some Galois cohomology groups for elliptic curves,
\newblock in: Elliptic Curves, Modular Forms and Iwasawa Theory (ed. D. Loeffler and S. L. Zerbes), Springer Proc. in Math. and Stat., {\bf 188} (2017) 373-399.



\bibitem{dmc} D. Macias Castillo,
\newblock Congruences for critical values of higher derivatives of twisted Hasse-Weil $L$-functions, II,
\newblock Acta Arith. {\bf 195} (2020), no. 7, 327-365.








\bibitem{mtbi} B. Mazur, J. Tate,
\newblock Canonical height pairings via biextensions,
\newblock In: `Arithmetic and Geometry' vol. 1,
 Prog. Math. {\bf 35} (1983) 195-237.

\bibitem{mt} B. Mazur, J. Tate,
\newblock Refined Conjectures of the Birch and Swinnerton-Dyer
Type, \newblock Duke Math. J.  {\bf 54} (1987) 711-750.




\bibitem{Neu} J.~Neukirch, 
\newblock Algebraic Number Theory,
\newblock Springer-Verlag, 1999.











\bibitem{SS03} E.F.~Schaefer, M.~Stoll,
\newblock How to do a p-descent on an ellitpic curve,
\newblock Transactions of the AMS, {\bf 356} (2003) 1209-1231.


\bibitem{Sch15a} E.F.~Schaefer, 
\newblock Computing a Selmer group of a Jacobian using functions on the curve,
\newblock Mathematische Annalen, {\bf 310} (1998) 447-471.

\bibitem{Sch15b} E.F.~Schaefer, 
\newblock Computing a Selmer group of a Jacobian using functions on the curve,
\newblock ArXiv e-prints, July 2015.









\bibitem{kst} K.-S. Tan,
\newblock $p$-adic pairings, 
\newblock Contemp. Math. {\bf 165} (1994) 111-121.








\bibitem{Visse} E. Visse,
\newblock Calculating the Tate local pairing for any odd prime number,
\newblock ArXiv e-prints, October 2016.



\bibitem{yakovlev} A. V. Yakovlev,
\newblock Homological definability of $p$-adic representations of groups with cyclic Sylow $p$-subgroup,
\newblock An. St. Univ. Ovidius Constan\c{t}a {\bf 4} (1996) 206-221.
\end{thebibliography}
\end{document}